\documentclass[reqno]{amsart}
\usepackage{mathrsfs}
\usepackage{amsmath}
\usepackage{amsthm}
\usepackage{amsfonts}
\usepackage{amssymb}
\usepackage{url}
\usepackage{enumerate}
\usepackage[pdftex,bookmarks=true]{hyperref}
  \usepackage[usenames, dvipsnames]{color}
\usepackage{verbatim}
\usepackage{graphicx}

\usepackage{pdfsync}

\newcommand\N{{\mathbb{N}}}
\newcommand\R{{\mathbb{R}}}
\newcommand\C{{\mathbb{C}}}

\newcommand\E{{\mathbb{E}}}

\newcommand\Var{{\operatorname{Var}}}
\newcommand\Cov{{\operatorname{Cov}}}

\newcommand\ep{\varepsilon}

\newcommand\la{\lambda}

\newcommand\Bd{{\mathbf d}}

\newcommand\BB{{\mathbf B}}

\newcommand\BD{{\mathbf D}}

\newcommand\BH{{\mathbf H}}

\newcommand\BN{{\mathbf N}}

\newcommand\BP{{\mathbf P}}

\newcommand\CF{{\mathcal F}}

\newcommand\CH{{\mathcal H}}

\newcommand\CK{{\mathcal K}}
\newcommand\CL{{\mathcal L}}

\newcommand\CT{{\mathcal T}}

\newcommand\rfl {\rfloor}
\newcommand\lfl {\lfloor}

\newcommand\supp{\mathbf{supp}}

\newcommand\eps{\varepsilon}

\newcommand\bs{\backslash}

\newcommand\tensor{\otimes}

\theoremstyle{plain}
  \newtheorem{theorem}[subsection]{Theorem}

    \newtheorem{proposition}[subsection]{Proposition}
  
  \newtheorem{lemma}[subsection]{Lemma}
  \newtheorem{corollary}[subsection]{Corollary}

  \newtheorem{remark}[subsection]{Remark}
  
  \newtheorem{claim}[subsection]{Claim}

\theoremstyle{definition}

\begin{document}

\title[CLT for roots of  orthogonal polynomials]{Central Limit Theorem for the number of real roots of random orthogonal polynomials}

 \author{Yen Do, Hoi H. Nguyen, Oanh Nguyen,  and Igor E. Pritsker}

\address{Department of Mathematics\\ The University of Virginia\\ 141 Cabell Drive, Charlottesville, VA 22904, USA}
\email{yendo@virginia.edu}

\address{Department of Mathematics\\ The Ohio State University \\ 231 W 18th Ave \\ Columbus, OH 43210 USA}
\email{nguyen.1261@osu.edu}

\address{Division of Applied Mathematics\\ Brown University\\  Providence, RI 02906, USA}
\email{oanh\_nguyen1@brown.edu}

\address{Department of Mathematics, Oklahoma State University, Stillwater, OK 74078, USA}
\email{igor@math.okstate.edu}

 \maketitle
%\subjclass[2010]{15B52, 60B20}

\begin{abstract} In this note we study the number of real roots of a wide class of random orthogonal polynomials with gaussian coefficients. Using the method of Wiener Chaos we show that the fluctuation in the bulk is asymptotically gaussian, even when the local correlations are different. 
\end{abstract}
\section{Introduction}

Let $\mu$ be a compactly supported Borel measure on the real line, and assume that the support of $\mu$ contains infinitely many points.  Consider the  polynomials
$$p_n(x) = \gamma_n x^n +\dots$$
with $\gamma_n>0$ for all $n\ge 0$,  such that
$$\int p_i(x) p_j(x) d\mu(x) =\delta_{ij}, \quad \text{for all }i,j \ge 0.$$
The study of the roots of these orthogonal polynomials has a long and rich history,  and we invite the  reader to classics such as \cite{Freud, ST} for  more details. In this note,  we are interested in several probabilistic aspects of the number of real roots for the following random polynomial, chosen from the span of $p_1(x),\dots, p_n(x)$:
$$H_n(x) = \sum_{j=0}^n \xi_j p_j(x).$$
We shall assume that $\xi_j$ are independent standard Gaussian random variables throughout the note.

Let $N_n([a,b])$ denote the number of real roots of $H_n$ in a given interval $[a,b]$.   Understanding the leading asymptotics for the expected value of $N_n$ was the subject of many prior investigations in this direction, and let us mention several results of this nature.

Around 1971,  Das \cite{Das} considered random Legendre polynomials (corresponding to the measure $d\mu (x) = dx$ on $[-1, 1]$) and found that the leading asymptotics of $\E N_n([-1, 1])$ is $n/\sqrt{3}$.  Wilkins \cite{Wilkins1, Wilkins2} later estimated the error term in this asymptotic relation. Farahmand \cite{Fa1,Fa2,Fa3} also considered the expected number of  level crossings for Legendre polynomials where the coefficients $\xi_j$'s may have dependent Gaussian distributions.

For random Jacobi polynomials (where $\mu(x) =(1-x)^a (1+x)^b \textbf{1}_{(-1,1)}$), Das and Bhatt \cite{DB} established that $\E N_n([-1, 1])$ has the same leading asymptotics $n/\sqrt{3}$.

These results were generalized by Lubinsky, Pritsker and Xie \cite{LPX1,LPX2} to much more general classes of random orthogonal polynomials.  In particular, they showed that the first term in the asymptotics for $\E N_n(\R)$ for many random polynomials remains the same.

All the results mentioned above were for Gaussian polynomials and were proved via the celebrated Kac-Rice formula.  We note in passing that the first moment statistics (and higher order moments) of the number of real roots of these polynomials are also universal in terms of randomness (when $\xi_j$ are not necessarily Gaussian), see  a recent joint work with V. Vu of the first author and the third author  \cite{DONgV}; however we will not focus on this aspect in this note.

It is well known that for $\mu(x)=(1-x)^{-1/2}(1+x)^{-1/2}1_{(-1,1)}$ we obtain the Chebyshev orthogonal polynomials (of the first type),  a special case of Jacobi random orthogonal polynomials.  The Chebyshev polynomials satisfy $p_n(\cos(x))=\cos(nx)$ and thus one may equivalently consider the random trigonometric polynomials
$$H_n(x)=\sum_{j=0}^n \xi_j \cos(j x), \quad x\in [0,2\pi].$$
This class of random functions was considered by Dunnage \cite{Dun} who showed that $\E N_n([0,2\pi])$ is asymptotically equal to $2n/\sqrt{3}$. In \cite{Qualls}, Qualls considered a slightly different class of trigonometric polynomials  (now known as the stationary trigonometric polynomials) \begin{equation}\label{eq:trig}
H_n(x)=\sum_{j=0}^n \xi_{j1} \cos(j x) + \xi_{j2} \sin(j x), \quad x\in [0,2\pi]
\end{equation}
and showed that $\E N_n([0,2\pi])$ is also asymptotically equal to $2n/\sqrt{3}$.

It is well-known in the subject that, for the number of real roots, asympotics for the variance  are much harder to handle than  the expected value:  for the variance,  one typically has to establish some further cancellation in the applications of the Kac-Rice formula. For the stationary random trigonometric model \eqref{eq:trig},  Bogomolny, Bohigas and Leboeuf \cite{BBL} argued that $\Var( N_n([0,2\pi]))$ is asymptotically $cn$, and this was verified by Granville and Wigman \cite{GW}, and subsequently by Aza\"{i}s and Le\'{o}n \cite{AL} via a different method, with an explicit formula for $c$.  The variance for the classical trigonometric models considered by Dunnage was computed in \cite{SS} by Su and Shao, and also by Aza\"{i}s, Dalmao and Le\'{o}n in \cite{ADL}. More recently,  Lubinsky and Pritsker \cite{LP} were able to provide a general method to compute the variances for many important cases of random orthogonal polynomials. Their results are summarized below.

\begin{theorem}[Mean and variance for roots of random orthogonal polynomials in the bulk]\cite{LPX1,LPX2,LP}\label{thm:var} Let $\mu$ be a measure with compact support on the real line, that is regular in the sense of Stahl, Totik, and Ullmann (that is $\lim_{n\to \infty} \gamma_n^{1/n} = \frac{1}{cap(\supp(\mu))}$, where $cap$ denotes the logarithmic capacity of $\supp(\mu)$). Let $\omega$ denote the Radon-Nikodym derivative of the equilibrium measure for the support of $\mu$. Let $[a', b']$ be a subinterval in the support of $\mu$, such that $\mu$ is absolutely continuous there, and its Radon-Nikodym derivative $\mu'$ is positive and continuous there. Assume moreover, that
$$\sup_{n\ge 1} \|p_n\|_{L_\infty[a',b']}<\infty.$$
Then if $[a,b] \subset (a',b')$, we have
$$\E N_n([a,b]) =  \left(\frac{\nu_\CK([a,b])}{\sqrt{3}}+o(1)\right)n,$$
where $\nu_\CK$ is the equilibrium measure of the support $\CK$ of $\mu$ in the sense of logarithmic potential theory.\

Furthermore,
$$\lim_{n\to \infty}\frac{1}{n} \Var(N_n([a,b])) = c \int_{a}^b \omega(y) dy$$
where $c$ is an explicit absolute positive constant (independent of $\CK, \mu$).
\end{theorem}
In what follows, for convenience we write
$$c_{a,b}:=c \int_{a}^b \omega(y) dy.$$
%\HC{We replaced $c_v n$ by $c_{a,b} n$ throughout for better dependence.}

We note that the constant $c$ is defined via the normalized sinc function sinc$(x)=\sin(\pi x)/(\pi x)$ in a rather complicated way, see \cite{LP} for details.  We also recall that $\nu_\CK$ minimizes the energy $I[\nu] = - \int \int  \log|z-t| d \nu(t) d\nu(z)$ among all probability measures $\nu$ with support on $\CK$, see \cite{Ra} and \cite{ST} for more details on properties of equilibrium measures and other facts from potential theory. It remains an interesting problem to extend the above result to the edges covering $[a',b']$.

After establishing the leading asymptotics for the variance,   one is naturally interested in another interesting and important direction, namely the  limiting distribution of the  fluctuation of $N_n[(a,b)]$ around  its mean (also known as the standardization of $N_n([a,b])$).  It was shown in \cite{GW} by  Granville and Wigman  that for the random trigonometric polynomials \eqref{eq:trig}, the limiting  distribution  for (the fluctuation of) $N_n([0,2\pi])$ is Gaussian, and this phenomenon is often refered to by various authors as asymptotic normality of $N_n([0,2\pi])$.  Granville--Wigman's result was later re-established in \cite{AL} by Aza\"{i}s and Le\'{o}n using a totally different and very powerful method.  Using this method,  Aza\"{i}s, Dalmao and Le\'{o}n  \cite{ADL}  also prove   that the   limiting  distribution   for $N_n([0,2\pi])$ for the random cosine polynomials is Gaussian, thus showed that after standardization the limiting distribution of the number of real roots of the random Chebyshev polynomials also obey the Gaussian law.

In this note, we show that the limiting   distribution for the number of real roots for the random polynomials considered in Theorem \ref{thm:var} also obeys the Gaussian law, hence extending the phenomenon established in \cite{ADL} (for Chebyshev polynomials) to far more general random polynomial ensembles.

%In what follows, for convenience, we let $N$ denote $N_n([a,b])$, and for brevity we write $\E N =(\bar{c}+o(1))n$ and $\Var(N) = (c_{a,b}+ o(1)) n$.

\begin{theorem}[Central Limit Theorem   in the bulk,  our main result]\label{thm:CLT} With the same assumptions and notations as in Theorem \ref{thm:var}, we have
$$\frac{N_n([a,b])-\E N_n([a,b])}{\sqrt{c_{a,b} n}}  \xrightarrow{d} \BN (0,1).$$
\end{theorem}

This result basically resolves Problem 1.5 from \cite{list}. There have been exciting developments regarding asymptotic normality for the number of real roots. These include, beside the results of \cite{ADL, AL,GW} mentioned above, the results \cite{DV,NS} for random Weyl polynomials and random Weyl series, \cite{Mas, ONgV} for random Kac polynomials and generalization,  and \cite{AnL,Dal} for random elliptic polynomials, which were proved using  very different methods.  For the stationary trigonometric polynomials,  \cite{GW} was able to reduce to a model where the correlations over far-apart points vanish, and then used a result of Berk for sum of long-range independent terms,  while for the crossings of these polynomials \cite{AL} computed all the central moments rather precisely,  For the Kac polynomials and generalizations, \cite{ONgV,Mas}  used comparison methods,    while for the Weyl polynomials and series \cite{DV}  used the method of cumulants based on a very fine understanding of the correlation functions (see also \cite{NS}).  Finally,  to handle the classical random trigonometric polynomials \cite{Dal, ADL} used the Wiener chaos decomposition,  and this is the method we will be using in this note. Among other things, one highlight of our work is that, because of the nature of the random orthogonal polynomials we are working with (see Section \ref{sec:proof:bounds}), our rescaled process $\overline{T}_n$ (to be defined below) does not have global limit, but over microscopic intervals it does converge to gaussian stationary processes. Interestingly, these processes are not necessarily the same unless $\omega(.)$ is a constant, see Section \ref{sect:tail} for further details.

Before proving our result, we deduce a few examples below.
%\HO{Please add.}

\begin{corollary} [Szeg\H{o} condition] Let $\mu$ be a measure supported on $[-1,1]$ satisfying the Szeg\H{o} condition
	$$\int_{-1}^{1} \log \mu'(x)\frac{dx}{\pi \sqrt {1-x^{2}}}>-\infty.$$
	Let $[a', b']$ be a subinterval of $(-1, 1)$, in which $\mu$ is absolutely continuous while $\mu'$ is positive and continuous in $[a', b']$. Assume moreover that its local modulus of continuity
	$$\Omega(t) = \sup\{|\mu'(x) - \mu'(y)|: x, y\in [a', b'] \text{ and } |x-y|\le t\}, \quad t>0,$$
	satisfies the Dini-Lipschitz condition
	$$\int_{0}^{1} \frac{\Omega(t)}{t}dt<\infty.$$
	Then for all $[a, b]\subset (a', b')$, $N_n(a, b)$ satisfies the CLT.
	
	In particular, when $\mu$ is the Legendre weight $\mu' = 1$, $N_n(a, b)$ satisfies the CLT for any $[a, b]\subset (-1, 1)$.
\end{corollary}

Following \cite{LP}, beyond Szeg\H{o} condition, we have

\begin{remark} Theorem \ref{thm:CLT} also applies to the following classes of measures that do not necessarily satisfy Szeg\H{o} condition:
\begin{itemize}
\item Measures associated with weights supported on several disjoint interval that satisfy the smoothness conditions in the classic paper of Widom \cite{Widom}.
\vskip .1in
\item Measures associated with exponential weights satisfying conditions in \cite{LevinLubinsky}. For example,
$$\mu'(x) = \exp(-\exp_k(1-x^{2})^{-\alpha}), x\in (-1, 1),$$
where $\alpha>0$ and $\exp_k(\cdot) = \exp(\exp(\dots \exp(\cdot)))$ denotes the $k$-th iterated exponential.
\end{itemize}
\end{remark}

For future directions, we hope to study the variance (as mentioned above) and the asymptotic normality of the number of roots for the {\it entire} interval $[a',b']$. There are technical problems here, especially in the analysis side as we will not have Lemma \ref{lemma:3.2} and Lemma \ref{lemma:3.3} near the edges $a',b'$. In another direction, it seems interesting to extend the variance estimate and CLT fluctuation to different types of randomness,  and this is left for further studies.

\subsection{Some preparations} \label{prep}

Here we fix some notations that will be used throughout the proof.  Define
$$T_n(t) = H_n(t/n) = \sum_{j=0}^n \xi_j p_j(t/n),\quad  t \in [na,nb]$$
and
$$\widetilde{T}_n(t) =\frac{1}{\sqrt{n}} \sum_{j=0}^n \xi_j p_j(t/n).$$
Hence, instead of finding the roots of $H_n$ in $[a,b]$,  we will be working with the number of roots $N ([na,nb])$ of $T_n$ (equivalently, $\widetilde{T}_n$). We will also use the following notations from \cite{LP} for the reproducing kernel associated with orthogonal polynomials:
\begin{equation}\label{K}
K_n(x,y) = \sum_{j=0}^{n} p_{j}(x) p_{j}(y),
\end{equation}
and for nonnegative integers $l,m$, its derivatives are denoted by
\begin{equation}\label{Klm}
K_n^{(l,m)}(x,y) = \sum _{j=0}^{n} p_{j}^{(l)}(x) p_{j}^{(m)}(y).
\end{equation}%

The correlation between $\widetilde{T}_n(t), \widetilde{T}_n(s)$ for any $t,s$ is
$$r_n(s,t) = \frac{1}{n}\sum_j p_j(s/n) p_j(t/n) = \frac{1}{n} K_n\left(s/n,t/n\right).$$
thus in particular the variance of $T_n(t)$ is
\begin{align}
\label{e.r0}
V_n(t)^2=r_n(t,t) &= \frac{1}{n}\sum_j (p_j(t/n))^2 = \frac{1}{n} K_n\left(t/n,t/n\right).
\end{align}
Observe that
\begin{align}
\label{e.r1}
\frac{\partial r_n(t,t)}{\partial t} &=  \frac{2}{n^2}\sum_j p_j(t/n) p_j'(t/n) = \frac{2}{n^2} K_n^{(0,1)}(t/n,t/n),\\
\label{e.r11}
\frac{\partial^2 r_n(s,t)}{\partial s \partial t} &=  \frac{1}{n^3}\sum_j p_j'(s/n) p_j'(t/n) = \frac{1}{n^3} K_n^{(1,1)}(s/n,t/n).
\end{align}

We denote the standardization of $\widetilde{T}_n$ by $\overline{T}_n$:
$$\overline{T}_n(t) = \frac{\widetilde{T}_n(t)}{V_n(t)}.$$
Let $\overline{r}_n(s,t)$ be the correlation of $\overline{T}_n(s)$ and $\overline{T}_n(t)$,  then
\begin{align}\label{eqn:rb}
\overline{r}_n(s,t) = \E (\overline{T}_n(s)  \overline{T}_n(t)) = \frac{r_n(s,t)}{\sqrt{r_n(s,s) r_n(t,t)}} = \frac{K_n(s/n,t/n)}{\sqrt{K_n(s/n,s/n) K_n(t/n,t/n)}}.
\end{align}
For brevity, we let
$$S_1(t):=\overline{T}_n(t) = \sum_j \xi_j q_j(t) = \frac{1}{\sqrt{n}} \sum_{j=1}^n \xi_j \frac{p_j(t/n)}{\sqrt{r_n(t,t)}}.$$
Then $S_1'(t) = \sum_j \xi_j q_j'(t)$,
$$\E S_1(s) S_1'(t) = \sum_j q_j(s) q_j'(t) = \frac{\partial \bar{r}_n(s,t)}{\partial t}\quad \mbox{ and }\quad \E S_1'(s) S_1'(t) = \sum_j q_j'(s) q_j'(t) = \frac{\partial^2 \bar{r}_n(s,t)}{\partial t \partial s}.$$
In particular,
$$\E((\overline{T}_n'(t))^2)= \E ((S_1'(t))^2) = \frac{\partial^2 \bar{r}_n(s,t)}{\partial t \partial s}|_{s=t}.$$

For brevity,  we  let $v_n(t)$ be the standard deviation of $\overline{T}'_n(t)$,  namely
\begin{equation}\label{def:vn}
v_n(t) = \sqrt{ \frac{\partial^2 \bar{r}_n(s,t)}{\partial t \partial s}|_{s=t}}=\sqrt{\E((\overline{T}_n'(t))^2)}.
\end{equation}

Let $\CT_n'(t)$ denote the standardization of $\overline{T}'_n(t)$:
$$\CT_n'(t) :=\frac{\overline{T}'_n(t)}{v_n(t)} = \sum_j \xi_j \frac{q_j'(t)}{v_n(t)} = \frac{1}{\sqrt{n}} \sum_{j=1}^n \xi_j \frac{1}{v_n(t)}\left[\frac{p_j(t/n)}{\sqrt{r_n(t,t)}}\right]'.$$
We stress that $\CT_n'(t)$ is not the derivative of $\CT_n(t)$,  which is not even defined here.

Note that for each $t$ by definition $\overline{T}_n(t)$ and $\CT_n'(t)$ are independent because they are jointly Gaussian and their correlation is 0. In general, we will define the correlation between $\overline{T}_n(s)$ and $\CT_n'(t)$ by
\begin{equation}\label{def:tilde:r'}
\widetilde{r}'_n(s,t) := \Cov(\overline{T}_n(s), \CT_n'(t) ) = \frac{1}{n}  \frac{1}{v_n(t)} \sum_j  \frac{p_j(s/n)}{\sqrt{r_n(s,s)}}  \left[\frac{p_j(t/n)}{\sqrt{r_n(t,t)}}\right]'
\end{equation}
(here we abused notations a bit, since $\tilde{r}'_n(s,t)$ is not a derivative of any function) and the correlation between $\CT'_n(s)$ and $\CT_n'(t)$ by
\begin{equation}\label{def:tilde:r''}
\widetilde{r}''_n(s,t) := \Cov(\CT_n'(s) , \CT_n'(t) ) =  \frac{1}{n}   \frac{1}{v_n(s) v_n(t)} \sum_j \left[\frac{p_j(s/n)}{\sqrt{r_n(s,s)}}\right]' \left[\frac{p_j(t/n)}{\sqrt{r_n(t,t)}}\right]'.
\end{equation}

\subsubsection{Reformulation   using the reproducing kernel and its derivatives.}
Our goal in this section is to reformulate $v_n$,  $\bar{r}_n$ and its derivatives in terms of $K_n$ and its derivatives.

We start with $v_n(t)$.  We first evaluate the mixed derivative of $\bar{r}_n(s,t)$:
\begin{align*}
\frac{\partial^2 \bar{r}_n(s,t)}{\partial t \partial s} &= \frac{\partial}{\partial t} \Big(\frac{\partial}{\partial s} [r_n(s,t)(r_n(s,s))^{-1/2} (r_n(t,t))^{-1/2}]\Big)\\
%&=\frac{\partial}{\partial t} \left[\frac{\partial r_n(s,t)}{\partial s} (r_n(s,s))^{-1/2} (r_n(t,t))^{-1/2} -\frac{1}{2} \frac{\partial r_n(s,s)}{\partial s} (r_n(s,s))^{-3/2} r_n(s,t) (r_n(t,t))^{-1/2}\right]  \\
&= \frac{\partial^2 r_n(s,t)}{\partial t \partial s} (r_n(s,s))^{-1/2} (r_n(t,t))^{-1/2}\\
 &-\frac{1}{2} \frac{\partial r_n(s,t)}{\partial s} (r_n(s,s))^{-1/2}  \frac{\partial r_n(t,t)}{\partial t} (r_n(t,t))^{-3/2} \nonumber  \\
&- \frac{1}{2}  \frac{\partial r_n(s,t)}{\partial t}\frac{\partial r_n(s,s)}{\partial s} (r_n(s,s))^{-3/2} (r_n(t,t))^{-1/2} \nonumber \\
& + \frac{1}{4} r_n(s,t) \frac{\partial r_n(s,s)}{\partial s} (r_n(s,s))^{-3/2} \frac{\partial r_n(t,t)}{\partial t} (r_n(t,t))^{-3/2}.
\end{align*}
Letting $s=t$ and using the definition of $K_n$ and \eqref{e.r0}, we obtain
\begin{align}\label{eqn:v_n}
v_n(t)^2&= \frac{\partial^2 \bar{r}_n(s,t)}{\partial t \partial s} \Big|_{s=t} \\
%&= \frac{1}{n^3}\sum_j (p_j'(t/n))^2 (r_n(t,t))^{-1} - 2  \left[\frac{1}{n^2}\sum_j  p_j(t/n) p_j'(t/n)\right]^2 (r_n(t,t))^{-2}   +   \left[\frac{1}{n^2}\sum_j  p_j(t/n) p_j'(t/n)\right]^2 (r_n(t,t))^{-2} \nonumber\\
&=  \frac{1}{n^3}\sum_j (p_j'(t/n))^2 (r_n(t,t))^{-1} -   \left[\frac{1}{n^2}\sum_j  p_j(t/n) p_j'(t/n)\right]^2 (r_n(t,t))^{-2}  \nonumber\\
&=  \frac{1}{n^2} \frac{K_n^{(1,1)}(t/n,t/n)}{K_n(t/n,t/n)} - \frac{1}{n^2} \left(\frac{K_n^{(0,1)}(t/n,t/n)}{K_n(t/n,t/n)}\right)^2.
\end{align}

Via explicit computation and using \eqref{e.r0} and \eqref{e.r1},  we obtain
 \begin{align}
\tilde{r}'_n(s,t)&= \frac{1}{n} \frac{1}{v_n(t)} \sum_j   \frac{p_j(s/n)}{\sqrt{r_n(s,s)}} \left[\frac{1}{\sqrt{r_n(t,t)}}  p_j(t/n)\right]' \nonumber \\
&= \frac{1}{n}  \frac{1}{v_n(t)} \sum_j  \frac{p_j(s/n)}{\sqrt{r_n(s,s)}} \left[\frac{1}{n}p_j'(t/n) (r_n(t,t))^{-1/2} - \frac{1}{2} p_j(t/n) \frac{\partial r_n(t,t)}{\partial t} (r_n(t,t))^{-3/2}\right] \nonumber \\
&= \frac{1}{n}  \frac{1}{v_n(t)}  \sum_j \Big[ (r_n(s,s))^{-1/2}  (r_n(t,t))^{-1/2} \frac{1}{n} p_j(s/n) p_j'(t/n)  \nonumber\\
&- \frac{1}{2} (r_n(s,s))^{-1/2} (r_n(t,t))^{-3/2}  \frac{\partial r_n(t,t)}{\partial t}   p_j(s/n)    p_j(t/n) \Big] \nonumber.
\end{align}
\begin{align}\label{eqn:r'}
\tilde{r}'_n(s,t)&= \frac{1}{n}  \frac{1}{v_n(t)\sqrt{r_n(s,s) r_n(t,t)}}  \sum_j \Big[ \frac{1}{n} p_j(s/n) p_j'(t/n)] - \frac{1}{2 r_n(t,t)} \frac{\partial r_n(t,t)}{\partial t}   p_j(s/n)    p_j(t/n) \Big] \nonumber \\
&= \frac{1}{v_n(t) \sqrt{K_n(s/n,s/n) K_n(t/n,t/n)}} \left(\frac{K_n^{(0,1)}(s/n,t/n)}{n} - \frac{K_n^{(0,1)}(t/n,t/n)}{n K_n(t/n,t/n)} K_n(s/n,t/n)\right).
\end{align}
Similarly,
\begin{align}\label{eqn:r''}
\tilde{r}''_n(s,t)&= \frac{1}{n}  \frac{1}{v_n(s) v_n(t)}  \sum_j  \left[\frac{p_j(s/n)}{\sqrt{r_n(s,s)}}\right]' \left[\frac{p_j(t/n)}{\sqrt{r_n(t,t)}}\right]' \nonumber \\
&= \frac{1}{n}  \frac{1}{v_n(s) v_n(t)}   \sum_j  \left[\frac{1}{n}p_j'(s/n) (r_n(s,s))^{-1/2} - \frac{1}{2} p_j(s/n) \frac{\partial r_n(s,s)}{\partial s} (r_n(s,s))^{-3/2}\right] \times \nonumber\\
& \times \left[\frac{1}{n} p_j'(t/n) (r_n(t,t))^{-1/2} - \frac{1}{2} p_j(t/n) \frac{\partial r_n(t,t)}{\partial t} (r_n(t,t))^{-3/2}\right] \nonumber \\
&=  \frac{1}{n} \frac{1}{v_n(s) v_n(t)\sqrt{r_n(s,s) r_n(t,t)}}  \sum_j    \Big[\frac{1}{n^2} p_j'(s/n) p_j'(t/n) \nonumber \\
& - \frac{1}{2}   (r_n(t,t))^{-1}  \frac{\partial r_n(t,t)}{\partial t} \frac{1}{n}  p_j'(s/n)    p_j(t/n)  \nonumber\\
&- \frac{1}{2} (r_n(s,s))^{-1}   \frac{\partial r_n(s,s)}{\partial s}  \frac{1}{n} p_j(s/n)    p_j'(t/n) \nonumber \\
&+ \frac{1}{4} (r_n(s,s))^{-1} (r_n(t,t))^{-1} \frac{\partial r_n(s,s)}{\partial s} \frac{\partial r_n(t,t)}{\partial t} p_j(s/n) p_j(t/n)\Big] \nonumber
\end{align}
therefore using \eqref{e.r0}, \eqref{e.r1}, and \eqref{e.r11}, we obtain
\begin{align}
\tilde{r}''_n(s,t)&=  \frac{1}{n}  \frac{1}{v_n(s) v_n(t)\sqrt{r_n(s,s) r_n(t,t)}}  \Big[\frac{K_n^{(1,1)}(s/n,t/n)}{n^2} - \frac{1}{2 r_n(t,t)} \frac{\partial r_n(t,t)}{\partial t} \frac{K_n^{(1,0)}(s/n,t/n)}{n}  \nonumber\\
&- \frac{1}{2 r_n(s,s)} \frac{\partial r_n(s,s)}{\partial s} \frac{K_n^{(0,1)}(s/n,t/n)}{n} + \frac{1}{4 r_n(s,s) r_n(t,t)} \frac{\partial r_n(s,s)}{\partial s} \frac{\partial r_n(t,t)}{\partial t} K_n(s/n,t/n) \Big] \nonumber \\
&=  \frac{1}{v_n(s) v_n(t) \sqrt{K_n(s/n,s/n) K_n(t/n,t/n)}}  \Big[\frac{K_n^{(1,1)}(s/n,t/n)}{n^2} \nonumber \\
&- \frac{K_n^{(0,1)}(t/n,t/n)}{n^2 K_n(t/n,t/n)} K_n^{(1,0)}(s/n,t/n) - \frac{K_n^{(0,1)}(s/n,s/n)}{n^2 K_n(s/n,s/n)} K_n^{(0,1)}(s/n,t/n) \nonumber \\
&+ \frac{K_n^{(0,1)}(t/n,t/n)}{n K_n(t/n,t/n)} \frac{K_n^{(0,1)}(s/n,s/n)}{n K_n(s/n,s/n)} K_n(s/n,t/n) \Big].
\end{align}

\vskip .1in

Using these reformulation,  we will prove the following crucial estimates  in Section \ref{sec:proof:bounds}.

\begin{lemma}\label{lemma:bounds}  With the same assumptions and notations as in Theorem~\ref{thm:var}.there exist constants $c,C>0$ that are independent of $n$, but might depend on $a,b,w,\mu$, such that the following hold:
\begin{itemize}
  \item[(i)] For all $t\in [na,nb]$,
  $$c \le v_n(t) \le C;$$
  \vskip .1in
  \item[(ii)] For all $s,t \in [na,nb]$,
$$|\overline{r}_n(s,t)| \le \frac{C}{|s-t| +1};$$
  \item[(iii)] For all $s,t \in [na,nb]$,
$$|\tilde{r}'_n(s,t)|  \le \frac{C}{|s-t| +1};$$
  \item[(iv)] For all $s,t \in [na,nb]$,
$$|\tilde{r}''_n(s,t)|  \le \frac{C}{|s-t| +1}.$$
\item[(v)] Finally, for any positive integer $k$ there exists a positive constant $C_k$  that is independent of $n$, but might depend on $a,b,\omega,\mu, k$, so that
$$\sum_{j=0}^n (q_j^{(k)}(t))^2 \le C_k,\quad t\in[na,nb],$$
where $q_j(t)=p_j(t/n)/\sqrt{n r_n(t,t)}$ were used in the definition of $S_1(t)$.
\end{itemize}
\end{lemma}

\section{Properties of orthogonal polynomials with compact support on $\R$: proof of Lemma \ref{lemma:bounds}} \label{sec:proof:bounds}

In this section, we mainly recall several key (deterministic) properties of orthogonal polynomials from \cite{LP}.  We shall use these properties  to justify Lemma \ref{lemma:bounds} later in the current section.

Our first ingredient is an estimate from \cite[Lemma 3.2]{LP} regarding the growth of lower order derivatives  of the reproducing kernel $K_n$ defined in \eqref{Klm}. We will use assume that the hypotheses of Theorem~\ref{thm:var} hold.
\begin{lemma}\label{lemma:3.2} For $l,m=0,1$ and $l=2,m=0$, and for all $n\ge 1$ and $x,y\in [a,b] \subset (a',b')$, we have
$$|K_n^{(l,m)}(x,y)| \le \frac{C n^{l+m}}{|x-y| + \frac{1}{n}}.$$
\end{lemma}
The above lemma can be shown using the Christoffel-Darboux formula
$$K_n(x,y) = \frac{\gamma_{n-1}}{\gamma_n} \frac{p_n(x) p_{n-1}(y) - p_{n-1}(x) p_{n}(y)}{x-y},$$
($\gamma_n$ is the leading coefficient of $p_n$) and  Bernstein's inequality for derivatives, $|P^{(j)}|_{L_\infty [a,b]} \le  C n^j \|P\|_{L_\infty[a',b']}$, which hold for any polynomial of degree $n$.

Additionally,  we will also need several  local limits for  the derivatives of the reproducing kernel $K_n$,  proved  in \cite[Lemma 3.3]{LP}.

\begin{lemma}\label{lemma:3.3} With the same notations as in Theorem \ref{thm:var},  let $[a', b']$ be a subinterval in the support of $\mu$, such that $\mu$ is absolutely continuous there, and its Radon-Nikodym derivative $\mu'$ is positive and continuous there.  Assume that $[a,b] \subset (a',b')$. Let $\omega(x)$ be the density of the equilibrium measure for $\CK=\textup{supp}\,\mu$, i.e., $d\nu_\CK(x)=\omega(x)\,dx.$ Then for $S(u) := (\sin \pi u) / (\pi u)$ and arbitrary non-negative integers $l,m$, we have
\begin{enumerate}[(a)]
\item Uniformly for $x\in [a,b]$ and $u,v$ in a compact subset of $\C$,
$$\lim_{n\to \infty} \frac{K_n^{(l,m)}(x + \frac{u}{n \omega(x)}, x + \frac{v}{n \omega(x)})}{K_n(x,x)} \left(\frac{1}{n \omega(x)}\right)^{l+m} = (-1)^m S^{(l+m)}(u-v). $$
\vskip .1in
\item Uniformly for $x\in [a,b]$,
$$\lim_{n\to \infty} \frac{K_n^{(l,m)}(x,x) \mu'(x)}{n^{l+m+1}} = \pi^{l+m} \omega^{l+m+1}(x) \tau_{l,m}$$
and
$$\lim_{n\to \infty} \frac{K_n^{(l,m)}(x,x) }{n^{l+m} K_n(x,x)} = (\pi \omega(x))^{l+m} \tau_{l,m},$$
where $\tau_{l,m} = (-1)^{(l-m)/2}/(l+m+1)$ if $l+m$ is even and $\tau_{l,m}=0$ otherwise.
\vskip .1in
\item In particular, uniformly for $x\in [a,b]$,
$$\lim_{n\to \infty} \frac{K_n^{(1,0)}(x,x) }{n^{2}}=0$$
and for $l=0,1$,
$$K_n^{(l,l)}(x,x) \ge C n^{2l+1}.$$
\end{enumerate}
\end{lemma}

% \HC{To Doron and Igor: below is the input from Igor (rewritten using the results cited above), please double check.}

\begin{proof}[Proof of Lemma \ref{lemma:bounds}]
In what follows, positive constants $A,B,C$ are independent of $n,s,t$ throughout the proof.  They may also depend on other parameters and may differ from one place to another.

(i) Using (b) of Lemma \ref{lemma:3.3}, we obtain that
\[
\lim _{n\rightarrow \infty }\frac{1}{n^2} \frac{K_{n}^{(1,1)}( x,x)}{K_{n}( x,x)} = A\, \omega(x)^2 \quad \mbox{and} \quad
\lim _{n\rightarrow \infty }\frac{1}{n} \frac{K_{n}^{(0,1)}( x,x)}{K_{n}( x,x)} = 0
\]
hold uniformly in $x\in[a,b]$. Here, $\omega(x)$ is the density of the equilibrium measure for $\textup{supp}\,\mu$, which is continuous and positive on $[a,b].$ Thus (i) follows directly from the representation \eqref{eqn:v_n} for $v_n(t)^2.$

(ii) It follows from Lemma \ref{lemma:3.2} and (c) of Lemma \ref{lemma:3.3} that
\[
|K_{n}(x,y)| \leq \frac{A}{|x-y| + 1/n} \quad \mbox{and} \quad K_n(x,x) \geq B\,n
\]
 for all $x,y\in[a,b]$, where $A,B>0$ do not depend on $x,y$. Setting $x=s/n$ and $y=t/n$, we estimate by \eqref{eqn:rb} that
 \[
 |\overline{r}_n(s,t)| \le \frac{A}{B(|s-t| +1)},\quad s,t \in [na,nb].
 \]

(iii) Since $K_n(t/n,t/n) \geq B\,n$ for all $t\in[na,nb]$, as in the proof of (ii), we obtain from (i) that
\[
v_n(t) \sqrt{K_n(s/n,s/n) K_n(t/n,t/n)} \ge c B\,n, \quad s,t \in [na,nb].
\]
Hence,
\[
|\tilde{r}'_n(s,t)| \le A \left(\frac{|K_n^{(0,1)}(s/n,t/n)|}{n^2} + \frac{|K_n^{(0,1)}(t/n,t/n)|}{n^2 K_n(t/n,t/n)} |K_n(s/n,t/n)|\right)
\]
by \eqref{eqn:r'}.

Using  Lemma \ref{lemma:3.2} we estimate that
\[
|K_{n}(s/n,t/n)| \leq \frac{An}{|s-t| + 1} \quad \mbox{and} \quad |K_n^{(0,1)}(s/n,t/n)| \leq \frac{B\,n^2}{|s-t| + 1},\quad s,t \in [na,nb].
\]
Furthermore, (b) of Lemma \ref{lemma:3.3} shows that
\begin{align} \label{aux1}
\lim _{n\rightarrow \infty }\frac{1}{n} \frac{K_{n}^{(0,1)}(t/n,t/n)}{K_{n}(t/n,t/n)} = 0
\end{align}
uniformly for $t\in[na,nb]$. Combining the last three estimates, we arrive at the statement of (iii).

(iv) We proceed in a similar way as in the proof of (iii), using (i) and the estimate $K_n(t/n,t/n) \geq B\,n$ for all $t\in[na,nb]$ to show that
\[
v_n(t) v_n(s) \sqrt{K_n(s/n,s/n) K_n(t/n,t/n)} \ge c^2 B\,n, \quad s,t \in [na,nb].
\]
Using Lemma \ref{lemma:3.2} as before, we have for $l,m=0,1$ that
\[
|K_n^{(l,m)}(s/n,t/n)| \leq \frac{C\,n^{l+m+1}}{|s-t| + 1},\quad s,t \in [na,nb].
\]
The latter two estimates applied to \eqref{eqn:r''} give
\begin{align*}
|\tilde{r}''_n(s,t)| \le \frac{A}{|s-t| + 1} \Bigg[ 1 &+ \frac{|K_n^{(0,1)}(t/n,t/n)|}{n K_n(t/n,t/n)} + \frac{|K_n^{(0,1)}(s/n,s/n)|}{n K_n(s/n,s/n)} \\
&+ \frac{|K_n^{(0,1)}(t/n,t/n)|}{n K_n(t/n,t/n)} \frac{|K_n^{(0,1)}(s/n,s/n)|}{n K_n(s/n,s/n)} \Bigg].
\end{align*}
Taking into account \eqref{aux1}, we complete the proof.

(v) This part of proof is based on similar ideas, so that we give a sketch. Using the chain rule, we have
$$q_j^{(k)}(t)  = \left( \frac{1}{\sqrt{n}}p_j(t/n) (r_n(t,t))^{-1/2}\right)^{(k)} =  \frac{1}{\sqrt{n}} \sum_{l=0}^k \binom{k}{l} \frac{1}{n^l}\, p_j^{(l)}(t/n)   \big((r_n(t,t))^{-1/2}\big)^{(k-l)}.$$
Setting $s_i(t) :=  \big((r_n(t,t))^{-1/2}\big)^{(i)}$ for compactness, we further rewrite
\begin{equation}\label{l_2_norm_q_j}
  \sum_{j=0}^n (q_j^{(k)}(t))^2 = \sum_{l_1,l_2=0}^k \binom{k}{l_1}\binom{k}{l_2} \frac{K_n^{(l_1,l_2)}(t/n,t/n)}{n^{l_1+l_2+1}}\, s_{k-l_1}(t) s_{k-l_2}(t).
\end{equation}
Since $r_n(t,t) =  \frac{1}{n} K_n(t/n,t/n)$, we obtain from part (b) of Lemma \ref{lemma:3.3} that $|r_n^{(j)}(t,t)|$ is bounded above, uniformly in $n\in\N$ and $t \in [na,nb]$, by a constant depending on $j, a,b,\omega,\mu$. Applying part (c) of Lemma \ref{lemma:3.3} on the lower bound of $r_n(t,t)$, we can now estimate $|s_i(t)|$ from above, uniformly in $n\in\N$ and $t \in [na,nb]$, by a constant depending on $i, a,b, \omega, \mu$. Using these estimates in \eqref{l_2_norm_q_j}, and applying Lemma \ref{lemma:3.3} (b) again to bound the terms $K_n^{(l_1,l_2)}(t/n,t/n)/n^{l_1+l_2+1},$ we obtain the desired result.
\end{proof}

% \HC{We added two results below for Section \ref{sect:tail}}

To complete the section, we cite here two more probabilistic results from \cite{LP} that will be useful later on.
\begin{lemma}\label{lemma:2.2} For any $(a_n,b_n) \subset (a',b')$, we have
$$\E N_n([a_n,b_n]) = \frac{1}{\pi} \int_{a_n}^{b_n} \rho_1(x)dx$$
and
$$\Var(N_n([a_n,b_n])) = \int_{a_n}^{b_n}\int_{a_n}^{b_n} (\rho_2(x,y) - \rho_1(x) \rho_1(y)) dx dy + \int_{a_n}^{b_n} \rho_1(x)dx,$$
where
$$\rho_1(x) = \frac{1}{\pi} \sqrt{ \frac{K_n^{(1,1)}(x,x) }{ K_n(x,x)} - \left(\frac{K_n^{(0,1)}(x,x) }{ K_n(x,x)}\right)^2}$$
and
$$\rho_2(x,y) = \frac{1}{\pi^2 \sqrt{\Delta}} \left(\sqrt{\Omega_{11}\Omega_{22}-\Omega^2_{12} } + \Omega_{12} \arcsin \left(\frac{\Omega_{11}}{\sqrt{\Omega_{11}\Omega_{22}}}\right)\right)$$
where
$$\Delta = K_n(x,x) K_n(y,y) - K_n^2(x,y)$$
and $\Omega_{ij}$ are the entries of the covariance matrix of $(P_n'(x), P_n'(y))$ conditioning on $P_n(x)=P_n(y)=0$,
$$\Omega_{11}(x,y)= K_n^{(1,1)}(x,x) - \frac{1}{\Delta} \Big( K_n^{}(y,y) (K_n^{(0,1)}(x,x))^2- 2 K_n^{}(x,y) K_n^{(0,1)}(x,x) K_n^{(0,1)}(y,x) +K_n^{}(x,x) (K_n^{(0,1)}(y,x))^2 \Big)$$
and
$$\Omega_{22}(x,y)= K_n^{(1,1)}(y,y) - \frac{1}{\Delta} \Big( K_n^{}(y,y) (K_n^{(0,1)}(x,y))^2- 2 K_n^{}(x,y) K_n^{(0,1)}(x,y) K_n^{(0,1)}(y,y) +K_n^{}(x,x) (K_n^{(0,1)}(y,y))^2 \Big)$$
and
\begin{align*}
\Omega_{12}(x,y)&= K_n^{(1,1)}(x,y) - \frac{1}{\Delta} \Big( K_n^{}(y,y) K_n^{(0,1)}(x,x)K_n^{(0,1)}(x,y)-  K_n^{}(x,y) K_n^{(0,1)}(x,y) K_n^{(0,1)}(y,x)\\
& -K_n^{}(x,y) K_n^{(0,1)}(x,x) K_n^{(0,1)}(y,y) + K_n^{}(x,x)  K_n^{(0,1)}(y,x) K_n^{(0,1)}(y,y)\Big).
\end{align*}
\end{lemma}
This result is \cite[Lemma 2.2]{LP}, which can be proved using Kac-Rice's formula. Note that by Lemma \ref{lemma:3.2} and (c) of Lemma \ref{lemma:3.3}, there exists a constant $C$ such that for all $x\in [a,b]$
\begin{equation}\label{eqn:rho1}
\rho_1(x) \le Cn.
\end{equation}
For the second-order correlations we can use Lemma \ref{lemma:3.3} to obtain (see \cite[Lemma 2.4]{LP}) the following.
\begin{lemma}\label{lemma:2.4} We have
\begin{enumerate}[(i)]
\item Uniformly for $u$ in compact subset of $\C\bs\{0\}$ and $x\in [a,b]$ and $y= x + \frac{u}{n\omega(x)}$, there is an explicit function $\Xi(.)$ (see \cite[(1.7)]{LP}) such that
$$\left(\frac{1}{n \omega(x)}\right)^2 (\rho_2(x,y) - \rho_1(x) \rho_1(y)) = \Xi(u) +o(1).$$
\item Let $\eta>0$, then there exists $C$ such that for $x\in [a,b]$ and $y= x + \frac{u}{n\omega(x)}, u \in [-\eta, \eta]$,
$$|\rho_2(x,y) - \rho_1(x) \rho_1(y))|\le Cn^2.$$
\end{enumerate}
\end{lemma}

\section{Wiener Chaos decomposition}\label{sect:Wiener}

In this section, we outline the method to prove Theorem \ref{thm:CLT}. These ideas seem to be standard, but we include them here for completeness. Our approach follows \cite{ADL}, see also \cite{PT} for a survey of related background.

Let $B=(B_\la)_{\la \in [0,1]}$ be a standard Brownian motion defined on some probability space $(\Omega,\CF, \BP)$ where $\CF$ is generated by $B$.

The process $(\widetilde{T}_n), n=1,2,\dots $,  described in Section \ref{prep}, can be defined on the same probability space by
$$\widetilde{T}_n(t) :=  \int_{0}^1 \gamma_n(t,\la) dB_\la := \int_{0}^1 \left[\sum_{j=1}^n p_j(t/n) 1_{((j-1)/n, j/n)}(\la)\right] dB_\la ,$$
where   $1_{((j-1)/n, j/n)}(.)$ is the indicator function of the  interval $((j-1)/n, j/n)$.

%\HC{Here we use Gaussian already.}

Let $\BH$ be the Hilbert space $\CL^2([0,1], \mathcal{B},\la)$ with the standard Borel $\sigma$-algebra,  Lebesgue measure,  and inner product. The map
$$h\mapsto \BB(h):= \int_{0}^1 h(\la) dB_\la$$
is an isometry between $\BH$ and $\CL^2(\Omega, \CF, \BP)$, which is also called an isonormal process associated to $\BH$ (see \cite[Chapter 8]{PT}). Using this terminology we can write
\begin{align*}
\widetilde{T}_n(t) &= \BB(\gamma_n(t,.)), \\
\bar{T}_n(t) &= \BB(h_n(t,.)), \mbox{ and } \\
 \CT_n'(t) &= \BB(h'_n(t,.)),
\end{align*}
where
\begin{equation}\label{eqn:h,h'}
h_n(t,\la):= \frac{\gamma_n(t,\la)}{V_n(t)} \mbox{ and } h_n'(t,\la):= \frac{\partial_t h_n(t,\la)}{\|\partial_t h_n(t,.)\|_2}.
\end{equation}

{\bf Multiple Wiener-Ito Integrals.} Now let $H_q(x)$ be the Hermite polynomial of degree $q$, defined by
$$H_q(x) = (-1)^q e^{x^2/2} \frac{\partial^q}{\partial t^q} e^{-x^2/2}.$$
It is well-known that these polynomials form a complete orthogonal system in $\CL^2(\R, \phi(dx))$, where $\phi$ is the Gaussian density.

For each $q\ge 1$, let $I_q^B$ denote the Multiple  Wiener-Ito integral with respect to the Brownian process $B$, obtained by extending linearly from the identity
 \begin{equation}\label{I:B}
 I_q^B(h^{\tensor q}) = H_q(\BB(h))
 \end{equation}
 where
$$h^{\tensor q}(\la_1,\dots, \la_q) := \prod h(\la_k).$$

$I_q^B$ can be viewed  as a linear isometry between the symmetric tensor product over $\CL_s^2([0,1]^q)$ (equiped with the norm $\sqrt{q!} \|.\|_{L^2[0,1]^q}$) and $\CL^2(\Omega, \CF, \BP)$  (for details see for instance \cite[Proposition 8.1.2]{PT}).

{\bf Wiener Chaos.} Next, one has the following decomposition
$$L^2(\Omega, \CF, \BP) = \oplus_{q=0}^\infty \CH_q,$$
where $\CH_q$, the $q$-th Wiener Chaos, is defined as the image of the linear functional $I_q^B$ over $\BH=L^2([0,1], \BB, d\la)$, and $\CH_0$ is the set of constants. More precisely, for any $F \in L^2(\Omega, \CF, \BP)$ there exists a unique sequence $\{f_q, q\ge 1\}$ with $f_q\in \CL_s^2([0,1]^q)$ such that
$$F-\E F = \sum_{q=1}^\infty I_q^B(f_q) $$
 (see for instance \cite[Theorem 8.2.1]{PT}).  We note that the constant component has been removed.

{\bf Chaos decomposition for the number of zeros, via the Kac-Rice formula.}
To prove Theorem~ \ref{thm:CLT},  we will use the following decomposition.
\begin{theorem}\label{thm:decomposition} The following holds in $L^2$-sense
$$\frac{N([na,nb]) - \E N}{\sqrt{c_{a,b}n}} =\sum_{q=2}^\infty I_q^{\overline{T}_n}([na,nb]),$$
where
$$ I_q^{\overline{T}_n}([na,nb]) = \frac{1}{\sqrt{c_{a,b}n}} \int_{na}^{nb} f_q(\overline{T}_n(s) , \CT_n'(s)) v_n(s)ds$$
and
$$f_q(x,y) = \sum_{l=0}^{\lfloor q/2 \rfloor} b_{q-2l} a_{2l} H_{q-2l}(x) H_{2l}(y), a_{2l}=\sqrt{\frac{2}{\pi}} \frac{(-1)^{l+1}}{2^l l! (2l-1)};  b_0=1, b_{2k} = \frac{H_{2k}(0)}{\sqrt{2\pi}(2k)!}, \mbox{ and } b_{2k+1}=0, \forall k.$$
\end{theorem}

%\HC{ Here we use $b_{2k} = \frac{H_{2k}(0)}{\sqrt{2\pi} (2k)!}$ as in \cite{AL} rather than $ b_{2k}= (-1)^{k} (2k-1)!!$ from \cite{ADL}.}

%\HY{This seems to be one of key steps for now, we might need a variant of \cite[Lemma 4, Lemma 5]{ADL} to prove the above theorem, can you double check, and perhaps move the proofs to Appendix \ref{sec:thm:decomposition} rather than here, if it works?}

%  Note: $(2k-1)!!=\frac{(2k)!}{2^k k!}$. So   $b_0=1$.

A similar decomposition was obtained in \cite{ADL,AL}. For completeness, in Section \ref{sect:decomp} we will present a proof of Theorem~\ref{thm:decomposition}  using a similar argument.  We remark that the same proof also shows that the following $L^2$-decomposition still holds for any subinterval $I$ of $[na, nb]$ % (without scaling by the standard deviation):
\begin{equation}\label{eqn:decomposition'}
N(I) - \E N(I) =\sum_{q=2}^\infty \int_I f_q(\overline{T}_n(s) , \CT_n'(s)) v_n(s)ds.
\end{equation}

%\HC{Added the above part for later reference.}

{\bf Chaining.}
In what follows, we will rewrite $I_q^{\overline{T}_n}([na,nb]) $ in terms of the Brownian motion $B$ and show that it belongs to the $q$th chaos.

Recall that $\overline{T}_n(t) = \BB(h_n(t,.))$ and $\CT_n'(t) = \BB(h'_n(t,.))$, and hence
$$H_{q-2l}(\overline{T}_n(t)) H_{2l}(\CT_n'(t)) =H_{q-2l}(\BB(h_n(t,.))) H_{2l}( \BB(h'_n(t,.))).$$
As remarked previously, $\overline{T}_n(t)$ and $\CT_n'(t)$ are orthogonal (i.e. independent), and so $h_n(t,.)$ and $h'_n(t,.)$ are also orthogonal by isometry. It then follows from \cite[Equation (6.4.17)]{PT} and \eqref{I:B} that
$$H_{q-2l}(\BB(h_n(t,.))) H_{2l}( \BB(h'_n(t,.))) = I_q^B( h_n(t,.)^{\tensor (q-2l)} \tensor h_n'(t,.)^{\tensor 2l}).$$
By the stochastic Fubini's theorem, we can express $I_q^{\overline{T}_n}([na,nb])$ as
\begin{align*}
I_q^{\overline{T}_n}([na,nb]) & = \frac{1}{\sqrt{c_{a,b}n}} \int_{na}^{nb} \sum_{l=0}^{\lfloor q/2 \rfloor} b_{q-2l} a_{2l} H_{q-2l}(\BB(h_n(s,.))) H_{2l}( \BB(h'_n(s,.)))v_n(s)ds \\
&=\frac{1}{\sqrt{c_{a,b}n}} \int_{na}^{nb} \sum_{l=0}^{\lfloor q/2 \rfloor} b_{q-2l} a_{2l}  I_q^B( h_n(s,.)^{\tensor (q-2l)} \tensor h_n'(s,.)^{\tensor 2l}) v_n(s)ds\\
&= I_q^B(g_q^{(n)}),
\end{align*}
where
\begin{equation}\label{eqn:g}
g_q^{(n)}(\mathbf{\la}) = \frac{1}{\sqrt{c_{a,b}n}} \int_{na}^{nb}  \sum_{l=0}^{\lfloor q/2 \rfloor} b_{q-2l} a_{2l} [h_{n}(s,.)^{\tensor (q-2l)} \tensor h_n'(s,.)^{\tensor 2l}](\la) v_n(s) ds; \la=(\la_1,\dots, \la_q).
\end{equation}
%\HC{switched from $I_q^B(g_q)$ to $ I_q^B(g_q^{(n)})$ to reflect $n$ dependence.}

{\bf Multidimensional CLT.} Last but not least,  to prove  Theorem \ref{thm:CLT} we will need the following CLT criterion extracted from \cite[Theorem 4]{ADL} and \cite[Theorem 11.8.3]{PT}. % \HC{not the finite version \cite[Theorem 11.8.1]{PT}.}
\begin{theorem}\label{thm:gau} Assume that
$$F_n = \sum_{q=1}^\infty I_q^B(f_q^{(n)})$$
such that
%\begin{itemize}
%\item
\begin{equation}\label{eqn:single}
\lim_{n\to \infty} \E (I_{q}^B (f_q^{(n)})^2) = \sigma_q^2<\infty;
\end{equation}
%\item
\begin{equation}\label{eqn:finiteness}
\sigma^2 = \sum_q \sigma_q^2 <\infty;
\end{equation}
%\item
and for each $q\ge 1$ and $p =0,\dots, q-1$ we have
\begin{equation}\label{eqn:cross}
\lim_{n \to \infty} \|f_q^{(n)} \tensor_p f_q^{(n)}\|_{L^2([0,1]^{2q-2p})} =0,
\end{equation}
%\item
\begin{equation}\label{eqn:tail}
\lim_{Q \to \infty} \limsup_{n \to \infty}\sum_{q\ge Q}  \E (I_{q}^B (f_q^{(n)})^2)=0.
\end{equation}
%\end{itemize}
Then $F_n$ converges to $\BN(0, \sigma^2)$ in distribution as $n\to \infty$.
\end{theorem}
We recall (see e.g. \cite{PT})  that for $0\le p\le q$ the $p$th contraction $ \tensor_p$ of functions  in the $q$th chaos can be computed using
\begin{align}
\label{eq:tensor}
\nonumber	
&(f_q^{(n)} \tensor_p f_q^{(n)} )(x_1,\dots, x_{q-p}, y_1,\dots, y_{q-p}) \\
&= \int_{[0,1]^p} f_q^{(n)}( z_1,\dots, z_p, x_1,\dots, x_{q-p}) f_q^{(n)}(z_1,\dots, z_p,y_1,\dots, y_{q-p})dz_1\dots dz_p.
\end{align}
In particular, in the special case $p=0$ this is simply the tensor product
$$f_q^{(n)} \tensor_0 f_q^{(n)} (x_1,\dots, x_{q}, y_1,\dots, y_{q}) =f_q^{(n)}(x_1,\dots, x_{q}) f_q^{(n)}(y_1,\dots, y_{q}).$$

%\HC{I dropped $q!$, it is likely that in \cite{PT} they did not normalize.}

% \HC{In \cite{ADL} it was
% $$f_q^{(k)} \tensor_p f_q^{(k)} (x_1,\dots, x_{q-p}, y_1,\dots, y_{q-p}) = \int_{[0,1]^p} f_q^{(k)}(x_1,\dots, x_{q-p}, z_1,\dots, z_p) f_q^{(k)}(y_1,\dots, y_{q-p}, z_1,\dots, z_p)dz_1\dots dz_p.$$
% But I think this is a mistake/typo, our correction fits with for instance \cite[6.2.10]{PT}.}

\section{Asymptotic decay of the contractions: Verification of Condition \eqref{eqn:cross}}\label{section:cross}
Using Theorem~\ref{thm:decomposition}, Theorem~\ref{thm:gau},  to show Theorem~\ref{thm:CLT} we need to verify the conditions in Theorem~\ref{thm:gau}.  As we'll see,  thanks to Theorem \ref{thm:var},   one of the main tasks is to  show that the contractions have vanishing norms in the limit (see Proposition \ref{prop:cov:zero} below),  and this is the main goal of the current section. 

Recall the definitions of $h_n(t,\lambda)$ and $h'_n(t,\lambda)$ from \eqref{eqn:h,h'}.  We first note that

\begin{claim}\label{claim:h} For $p\le i$, we have
$$h_n(s,.)^{\tensor i} \tensor_p h_n(t,.)^{\tensor i} = \overline{r}_n(s,t)^p h_n(s,.)^{\tensor i-p} \tensor_0 h_n(t,.)^{\tensor i-p},$$
$$h_n(s,.)^{\tensor i} \tensor_p h_n'(t,.)^{\tensor i} = \tilde{r}_n'(s,t)^p h_n(s,.)^{\tensor i-p} \tensor_0 h_n'(t,.)^{\tensor i-p},$$
and
$$h_n'(s,.)^{\tensor i} \tensor_p h_n'(t,.)^{\tensor i} = \tilde{r}_n''(s,t)^p h_n(s,.)^{\tensor i-p} \tensor_0 h_n'(t,.)^{\tensor i-p}.$$

\end{claim}
\begin{proof}  For the first identity,  letting $x=(x_1,\dots,x_p)$ we have
\begin{align*}
&h_n(s,.)^{\tensor i} \tensor_p h_n(t,.)^{\tensor i} (\la_1,\dots, \la_{i-p}, \la'_1,\dots, \la'_{i-p})\\
& = \int_{x_1,\dots,x_p}  h_n(s,.)^{\tensor i}(x_1,\dots,x_p, \la_1,\dots, \la_{i-p}) h_n(t,.)^{\tensor i}(x_1,\dots,x_p, \la'_1,\dots, \la'_{i-p}) dx\quad\text{by \eqref{eq:tensor}}\\
&= \Big(\int_{x_1,\dots,x_p}  h_n(s,x_1) \times \dots \times h_n(s,x_p)h_n(t,x_1) \times \dots \times h_n(t,x_p)   dx\Big) \times \\
&\times h_n(s,.)^{\tensor i-p}(\la_1,\dots, \la_{i-p}) h_n(t,.)^{\tensor i-p}(\la_1',\dots, \la_{i-p}') \\
&=  \left (\int_0^1  \frac{\gamma_n(s,\la)}{V_n(s)}  \frac{\gamma_n(t,\la)}{V_n(t)} d\la\right )^p \times   h_n(s,.)^{\tensor i-p}(\la_1,\dots, \la_{i-p})  h_n(t,.)^{\tensor i-p}(\la_1',\dots, \la_{i-p}')   \quad\text{by \eqref{eqn:h,h'}}\\
&=    \overline{r}_n(s,t)^p  \times h_n(s,.)^{\tensor i-p}(\la_1,\dots, \la_{i-p})  h_n(t,.)^{\tensor i-p}(\la_1',\dots, \la_{i-p}').
\end{align*}
Similarly,  to obtain the second and the third identities we   use the fact that
$$\int_0^1 h_n(s,x) h_n'(t,x)dx = \widetilde{r}'_n(s,t)$$
and
$$\int_0^1 h_n'(s,x) h_n'(t,x)dx = \widetilde{r}''_n(s,t).$$
%\HO{Please double check when possible}
\end{proof}

\begin{claim}\label{claim:tensor} For $0\le k,k'\le i$, let
$$T_{n, i, k, k'}(s, t):=\int h_n(s,.)^{\tensor i-k} \tensor_0 h_n'(s,.)^{\tensor k}(\la_1,\dots, \la_i) h_n(t,.)^{\tensor i-k'} \tensor_0 h_n'(t,.)^{\tensor k'}(\la_1,\dots, \la_i) d\la_1\dots \la_i.$$
%\HC{Switched from $T_{i, k, k'}(s, t)$ to $T_{n, i, k, k'}(s, t)$ to reflect dependence in $n$.}

Then
$$T_{n,i, k, k'}(s, t)=\begin{cases}
 \overline{r}_n(s,t)^{i-k}  (\widetilde{r}'_n(t, s))^{k- k' }  (\widetilde{r}''_n(s,t))^{k'} \text{ if } k\ge k',\\
 \overline{r}_n(s,t)^{i-k'}  (\widetilde{r}'_n(s,t))^{k'- k}  (\widetilde{r}''_n(s,t))^{k} \text{ if } k<k'.
\end{cases} $$

\end{claim}
\begin{proof} Without loss of generality, assume that $k\ge k'$. Then
\begin{align*}
&\int_{\la_1,\dots,\la_i} h_n(s,.)^{\tensor i-k} \tensor_0 h_n'(s,.)^{\tensor k}(\la_1,\dots, \la_i) h_n(t,.)^{\tensor i-k'} \tensor_0 h_n'(t,.)^{\tensor k'}(\la_1,\dots, \la_i) d\la_1\dots \la_i\\
& = \int_{\la_1,\dots, \la_{i-k}} h_n(s,\la_1) \dots h_n(s,\la_{i-k}) \times  h_n(t,\la_1) \dots h_n(t,\la_{i-k})  d\la_1\dots \la_{i-k} \\
& \times \int_{\la_{i-k+1},\dots, \la_{i-k'}} h_n'(s,\la_{i-k+1}) \dots h_n'(s,\la_{i-k'}) \times   h_n(t,\la_{i-k+1}) \dots h_n(t,\la_{i-k'}) d\la_{i-k+1} \dots d\la_{i-k'}\\
& \times \int_{\la_{i-k'+1},\dots, \la_{i}} h_n'(s,\la_{i-k'+1}) \dots h_n'(s,\la_{i}) \times   h_n'(t,\la_{i-k'+1}) \dots h_t'(s,\la_{i}) d\la_{i-k'+1} \dots d\la_{i}\\
&=  \overline{r}_n(s, t)^{i-k}  (\widetilde{r}'_n(t, s))^{k-k'}  (\widetilde{r}''_n(s,t))^{k'}.
\end{align*}
This completes the proof.
\end{proof}

In what follows we verify Condition \eqref{eqn:cross}.  
\begin{proposition}\label{prop:cov:zero} For each $g_q^{(n)}$ from \eqref{eqn:g}, with $q\ge 1$ we have
$$\lim_{n \to \infty} \|g_q^{(n)} \tensor_p g_q^{(n)} \|_{L^2([0,1]^{2q-2p})}^2=0 \mbox{ for each } p=1,\dots, q-1.$$
\end{proposition}
Let $q\ge 1$ and $0\le p\le q-1$.  Denoting $x=(x_1,\dots,x_p)$,  $\la=(\la_1,\dots,\la_{q-p})$ and $\la'=(\la'_1,\dots,\la'_{q-p})$,  we have
\begin{eqnarray*}
	&&\Big[g_q^{(n)} \tensor_p g_q^{(n)}\Big] (\la, \la') = \int_{x\in [0,1]^p}  g_q^{(n)} (x,\la) g_q^{(n)} (x,\la') dx\\
%	&=&   \int_{x\in [0,1]^p}  \left [\frac{1}{\sqrt{c_{a,b}n}} \int_{[na,nb]}  \sum_{l=0}^{q/2} b_{q-2l} a_{2l} \Big[h_n(s,.)^{\tensor (q-2l)} \tensor_0 h_n'(s,.)^{\tensor 2l}\Big](x,\la)  v_n(s) ds\right ] \times\\
%	&&\quad \left [\frac{1}{\sqrt{c_{a,b}n}} \int_{[na,nb]}  \sum_{l=0}^{q/2} b_{q-2l} a_{2l} \Big[h_n(t,.)^{\tensor (q-2l)} \tensor_0 h_n'(t,.)^{\tensor 2l}\Big](x,\la')  v_n(t) dt\right ] dx\\
	&=&    \frac{1}{c_{a,b}n} \int_{[na,nb]}\int_{[na,nb]}   v_n(s)v_n(t)\sum_{l,l'=0}^{q/2} b_{q-2l} a_{2l} b_{q-2l'} a_{2l'} \times \\
&&\times \int_{x\in [0,1]^p}  \Big[h_n(s,.)^{\tensor (q-2l)} \tensor_0 h'_n(s,.)^{\tensor 2l}\Big](x,\la)\times    \Big[h_n(t,.)^{\tensor (q-2l')} \tensor_0 h'_n(t,.)^{\tensor 2l'}\Big](x,\la') 
 dx    dsdt
\end{eqnarray*}
We note that if $q-2l > p$ then
\begin{eqnarray*}
&&\Big[h_n(s,.)^{\tensor (q-2l)} \tensor_0 h_n'(s,.)^{\tensor 2l}\Big](x,\lambda) \\ 
&=& h_n(s,.)^{\tensor p}(x)  \times  h_n(s,.)^{\tensor(q-2l-p)}(\lambda_1,\dots,\lambda_{q-p-2l}) \times h'_n(s,.)^{\tensor (2l)}(\la_{q-p-2l+1},\dots,\la_{q-p})
\end{eqnarray*}
and if $q-2l \le p$ then
\begin{eqnarray*}
&&\Big[h_n(s,.)^{\tensor (q-2l)} \tensor_0 h_n'(s,.)^{\tensor 2l}\Big](x,\lambda) \\ 
&=& h_n(s,.)^{\tensor (q-2l)}(x_1,\dots,x_{q-2l}) \times h'_n(s,.)^{\tensor(2l)}(x_{q-2l+1},\dots,x_p,\lambda_1,\dots,\lambda_{q-p}) 
\end{eqnarray*}

Consequently,  using Claim~\ref{claim:h} we obtain
\begin{eqnarray*}
&&\int_{x\in [0,1]^p} \Big[h_n(s,.)^{\tensor (q-2l)} \tensor_0 h'_n(s,.)^{\tensor 2l}\Big](x,\la)   \times   \Big[h_n(t,.)^{\tensor (q'-2l')} \tensor_0 h'_n(t,.)^{\tensor 2l'}\Big] (x,  \la') dx\\
&=&   R_{n,c,d}(s,t)^p   h_{n}(s,.)^{\tensor c} \tensor_0 h_n'(s,.)^{\tensor (q-p-c)}(\la)  \times h_{n}(t,.)^{\tensor d} \tensor_0 h_n'(t,.)^{\tensor (q-p-d)}(\la')
\end{eqnarray*}
where $c=c(q,p,l)=\max(q-2l-p,0)$ and $d=d(q,p,l)=\max(q-2l'-p,0)$, and 
$$R_{n,c,d}(s,t)=\begin{cases}\bar{r}_n(s,t), & \text{if $c,d>0$;}\\
\tilde{r}'_n(s,t), & \text{if $c>0$ and $d=0$;}\\
\tilde{r}'_n(t,s), & \text{if $c=0$ and $d>0$;}\\
\tilde{r}''_n(s,t), & \text{if $c=d=0$.}
\end{cases}
$$

% \HC{It seems that \cite[p. 814]{ADL} made a mistake here (implicitly), that $\bar{r}_n(s,t)$ should be $\tilde{r}''_n(s,t)$ because the $x$ affected the last coordinate. However if the $\tensor$ definition was as we corrected above, then this is fine.}

Hence
\begin{align*}
&\|g_q^{(n)} \tensor_p g_q^{(n)}\|_2^2 =    \frac{1}{(c_{a,b} n)^{2}}  \int_{ [na,nb]^4 }  v_n(s) v_n(t) v_n(s') v_n(t')  \\
 &\times \sum_{l,l', k, k'=0}^{q/2} \Big(b_{q-2l} a_{2l}  b_{q-2l'} a_{2l'}  b_{q-2k} a_{2k}  b_{q-2k'} a_{2k'} (R_{n,c,d}(s,t))^p  R_{n,c',d'}(s',t'))^p\times \\
&\times T_{n,p,q,l,l',k,k'}(s,t,s',t') \, dsdtds'dt'\Big),
\end{align*}
where
\begin{align*}
& T_{n,p,q,l,l',k,k'}(s,t,s',t')\\
& =\int_{\la_1,\dots, \la_{q-p}, \la_1',\dots, \la_{q-p}'} h_{n}(s,.)^{\tensor c} \tensor_0 h_n'(s,.)^{\tensor (q-p-c)}(\la_1,\dots, \la_{q-p}) h_{n}(t,.)^{\tensor d} \tensor_0 h_n'(t,.)^{\tensor (q-p-d)}(\la_1',\dots, \la_{q-p}') \\
& \times h_{n}(s',.)^{\tensor c'} \tensor_0 h_n'(s',.)^{\tensor (q-p-c')}(\la_1,\dots, \la_{q-p}) h_{n}(t',.)^{\tensor d'} \tensor_0 h_n'(t',.)^{\tensor (q-p-d')}(\la_1',\dots, \la_{q-p}') d\la_1 \dots d\la_{q-p}'\\
&= \int_{\la_1,\dots, \la_{q-p}} \bigg[h_{n}(s,.)^{\tensor c} \tensor_0 h_n'(s,.)^{\tensor (q-p-c)}(\la_1,\dots, \la_{q-p}) \nonumber\\
&\qquad \qquad\qquad\times h_{n}(s',.)^{\tensor d} \tensor_0 h_n'(s',.)^{\tensor (q-p-d)}(\la_1,\dots, \la_{q-p})\bigg] d\la_1 \dots d\la_{q-p}\\
&\times \int_{\la_1',\dots, \la_{q-p}'} \bigg[ h_{n}(t,.)^{\tensor c'} \tensor_0 h_n'(t,.)^{\tensor q-p-c'}(\la_1',\dots, \la_{q-p}') \\
&\qquad\qquad \qquad \times h_{n}(t',.)^{\tensor d'} \tensor_0 h_n'(t',.)^{\tensor (q-p-d')}(\la_1',\dots, \la_{q-p}') \bigg]d\la_1' \dots d\la_{q-p}'\\
&= T_{n,q-p, q-p-c, q-p-d}(s, s') T_{n,q-p, q-p-c', q-p-d'}(t, t'),
\end{align*}
where the $T$ terms are defined in Claim \ref{claim:tensor}.
Hence, by using Claim \ref{claim:tensor} and noting that all the coefficients have order $O_q(1)$, we have the following

\begin{claim}\label{claim:Spq} We can bound $(c_{a,b}n)^2\|g_q \tensor_p g_q\|_2^2$ by
\begin{align*}
&(c_{a,b}n)^2\|g_q^{(n)} \tensor_p g_q^{(n)}\|_2^2 =O_q\Big( \max_{i_1+i_2+i_3=q-p,i_1'+i_2'+i_3'=q-p}\int_{[na,nb]^4}v_n(s) v_n(t) v_n(s') v_n(t')\times\\
&\times \max (|\bar{r}_{n}(s,t)|,|\tilde{r}'_{n}(s,t)|,|\tilde{r}'_{n}(t,s)|,|\tilde{r}''_{n}(s,t)|)^p \max (|\bar{r}_{n}(s',t')|,|\tilde{r}'_{n}(s',t')|,|\tilde{r}'_{n}(t',s')|,|\tilde{r}''_{n}(s',t')|)^p \times \\
&\times  [|\bar{r}_n(s, s')|^{i_1} \max\{|\tilde{r}'_n(s, s')|, |\tilde{r}'_n(s', s)|\}^{i_2} |\tilde{r}''_n(s, s')|^{i_3}] \times  [|\bar{r}_n(t,t')|^{i_1'} \max\{|\tilde{r}'_n(t, t')|, |\tilde{r}'_n(t', t)|\}^{i_2'} |\tilde{r}''_n(t,t')|^{i_3'}]dsdtds'dt'\Big).
\end{align*}
\end{claim}

\begin{proof}[Proof of Proposition \ref{prop:cov:zero}] Let $S_{p,q,i_1,i_2,i_3,i_1',i_2',i_3'}$ be the integrand as above, we need to show that
\begin{equation}\label{eqn:S}
\frac{1}{n^2} S_{p,q,i_1,i_2,i_3,i_1',i_2',i_3'} = o(1).
\end{equation}
By Lemma \ref{lemma:bounds} we have
$$v_n(s),v_n(t), v_n(s'), v_n(t'), \bar{r}_n(s,t), \bar{r}_n(s',t'), \bar{r}_n(s,s'),\tilde{r}'_n(s,s'), \tilde{r}''_n(s,s'),\bar{r}_n(t,t'),\tilde{r}'_n(t,t'),\tilde{r}''_n(t,t')=O(1)$$
uniformly in $s,s',t,t'$. By Claim \ref{claim:Spq}, at least one of the $i_1,i_2,i_3$ must be at least 1, and similarly at least one of the $i_1',i_2',i_3'$ must be at least 1.  Hence by the bounds on the correlations from Lemma \ref{lemma:bounds}, it suffices to show
$$ \frac{1}{n^2} \int_{[na,nb]^4} \frac{1}{|s-t|+1}\frac{1}{|s'-t'|+1}  \frac{1}{|s-s'|+1}\frac{1}{|t-t'|+1}dsdtds' dt' =o(1).$$
Let $S$ denote the above integral.  If $p\ge 1$ then we have
\begin{eqnarray*}
S&\le& \int_{[na,nb]^4} \frac{1}{|s-t|+1}\frac{1}{|t-t'|+1}\frac{1}{|t'-s'|+1} dsdtds' dt'\\
&\le& \int_{na}^{nb}\int_{[-n(b-a), n(b-a)]^{3}} \frac{1}{|x|+1}\frac{1}{|y|+1}\frac{1}{|z|+1} dxdydz\,ds\\
&& \text{ by a change of variables $x = s-t, y=t-t', z = t'-s'$}\\
&=& 8 n(b-a) \left (\int _{0}^{n(b-a)}\frac{dx}{x+1}\right )^{3} = O(n\log^{3}n) = o(n^{2}).
\end{eqnarray*}
 
This completes the proof.
 \end{proof}

\section{Verification of Condition \eqref{eqn:single} and  Condition \eqref{eqn:finiteness}}\label{sect:single}
We first verify Condition \eqref{eqn:single}. By Mehler's formula \cite{AL}, we can write
\begin{eqnarray}\label{eqn:L_2:n}
&&\Var (I_q^B(g_q^{(n)})) =\E(I_q^B(g_q^{(n)})^2) \\
&=&\sum_{l,l'=0}^{q/2} b_{q-2l} a_{2l} b_{q-2l'}a_{2l'} \sum_{\Bd \in \BD_{q,2l,2l'}} \frac{(q-2l)! (2l)! (q-2l')! (2l')!} {d_1! d_2! d_3! d_4!}  \frac{1}{c_{a,b} n}   \int_{[na,nb]^2}  v_n(s) v_n(t) S_{n,q,\Bd}(s,t)ds dt, \nonumber
\end{eqnarray}
where $\BD_{q,2l,2l'}$ is the set of non-negative integral tuples $\Bd=(d_1,d_2,d_3,d_4)$ such that
\begin{equation}\label{eq:def:BD}
d_1+d_2=q-2l, d_3+d_4=2l, d_1+d_3=q-2l', d_2+d_4=2l'
\end{equation}
and for $\Bd=(d_1,d_2,d_3,d_4)\in \BD_{q,2l,2l'}$,
\begin{equation}\label{eq:def:S}
S_{n,q,\Bd}(s,t) := \bar{r}_n(s, t)^{d_1}  {\tilde{r}'_n(s, t)}^{d_2}  {\tilde{r}'_n(t, s)}^{d_3}{\tilde{r}''_n(s, t)}^{d_4}dsdt.
\end{equation}

% \HC{Here we used Mehler's formula, so there is no $T_{q,2l,2l'}$ but we have a similar quantity called $S_{n,\Bd}$. For the rest of this section I used Oanh's writing.}

As $q$ is fixed, for Condition  \eqref{eqn:single} it suffices to show the following

\begin{lemma}\label{lemma:L2:lim} For $q\ge 2$ and for any $l,l'$ and $\Bd \in \BD_{q,2l,2l'}$, we have
$$\lim_{n \to \infty} \frac{1}{c_{a,b} n}   \int_{[na,nb]}  \int_{[na,nb]}  v_n(s) v_n(t) S_{n,q,\Bd}(s,t)ds dt<\infty.$$
  \end{lemma}

\begin{proof} Let $\tau= t-s, \sigma = t+s$, we rewrite as
$$ \frac{1}{c_{a,b} n}   \int_{[na,nb]}  \int_{[na,nb]}  v_n(s) v_n(t) S_{n,q,\Bd}(s,t)ds dt=\frac{1}{2}\int_{-n(b-a)}^{n(b-a)} g_{n,q,\Bd}(\tau) d\tau$$
where
\begin{equation}\label{eqn:gtau}
g_{n,q,\Bd}(\tau) = \frac{1}{c_{a,b} n}  \int_{2na + |\tau|}^{2 nb - |\tau|} v_n((\sigma-\tau)/2)  v_n((\sigma+\tau)/2)  S_{n,q,\Bd}((\sigma-\tau)/2,(\sigma+\tau)/2) d\sigma.
\end{equation}
\begin{claim}\label{claim:tau:large}  There exists a constant $B$ depending on $a,b, \omega$ but not on $q$ such that
$$ |g_{n,q,\Bd}(\tau)| \le  \frac{B^q}{(|\tau|+1)^q}$$
uniformly in $n$ and $\tau$.
\end{claim}
\begin{proof} This follows from \eqref{eq:def:S} and Lemma \ref{lemma:bounds}, noting that the total number of the correlation factors in $S_{\Bd}$ (each bounded by $O(\frac{1}{\tau +1})$ in absolute value) is $q$.
\end{proof}

Now we consider $\tau$ fixed (independent of $n$). By changing of variable from $\sigma$ to $2n\theta$, we can write
\begin{eqnarray}
g_{n,q,\Bd}(\tau) &= &\frac{2}{c_{a,b}}  \int_{a + |\tau|/2n}^{b - |\tau|/2n}  v_n(n\theta+\tau/2) v_n(n\theta-\tau/2)  S_{n,q,\Bd}(n\theta-\tau/2,n\theta+\tau/2) d\theta \nonumber \\
&=& \frac{2}{c_{a,b}}  \int_{a  }^{b } \textbf{1}_{[a + |\tau|/2n, b - |\tau|/2n]}(\theta) f_{n,q,\Bd}(\tau, \theta)d\theta,\label{eq:gn:fn}
\end{eqnarray}
where
$$f_{n, q, \Bd}(\tau, \theta) = v_n(n\theta+\tau/2) v_n(n\theta-\tau/2)  S_{n,q,\Bd}(n\theta-\tau/2,n\theta+\tau/2).$$

By (a) of Lemma \ref{lemma:3.3} we have uniformly for $\theta\in [a, b]$ and $\tau$ in a compact set,
\begin{equation}\label{eqn:=}
\lim_{n\to \infty} \frac{K_n^{(l,m)}(\theta+ \tau/2n, \theta+\tau/2n)}{n^{l+m}K_n(\theta,\theta)} =  (-1)^m (\omega(\theta))^{l+m}S^{(l+m)}(0) = (-1)^m (\omega(\theta))^{l+m} \pi^{l+m} \tau_{l,m}
\end{equation}
and
\begin{equation}\label{eqn:+-}
\lim_{n\to \infty} \frac{K_n^{(l,m)}(\theta+ \tau/2n, \theta-\tau/2n)}{n^{l+m}K_n(\theta,\theta)} =  (-1)^m (\omega(\theta))^{l+m}S^{(l+m)}(\tau \omega(\theta)),
\end{equation}
where we recall that $\tau_{l,m} = (-1)^{(l-m)/2}/(l+m+1)$ if $l+m$ is even and $\tau_{l,m}=0$ otherwise.

Hence by the formula for $v_n$ from \eqref{eqn:v_n}, the following holds uniformly for $\theta\in [a, b]$ and $\tau$ in a compact set
% \HC{uniformly in $\theta$ and $\tau$}
\begin{align}
\lim_{n \to \infty} v_n^2(n\theta+\tau/2) &=\lim_{n \to \infty} \left[ \frac{1}{n^2} \frac{K_n^{(1,1)}(\theta+ \tau/2n, \theta+\tau/2n)}{K_n(\theta+ \tau/2n,\theta+ \tau/2n)} - \frac{1}{n^2} \left(\frac{K_n^{(0,1)}(\theta+ \tau/2n,\theta+ \tau/2n)}{K_n(\theta+ \tau/2n,\theta+ \tau/2n)}\right)^2 \right] \nonumber\\
&=\omega(\theta)^2 \pi^2/3, \label{eq:v:lim}
\end{align}
where we applied \eqref{eqn:=} to $(l, m) = (1, 1), (0, 0)$ and $(0,1)$.
Similarly,\begin{equation}\nonumber
\lim_{n\to \infty} v_n^2(n\theta-\tau/2) = \omega(\theta)^2 \pi^2/3.
\end{equation}
Furthermore, again by Lemma \ref{lemma:3.3}, and by \eqref{eqn:rb} with $s= n\theta-\tau/2$ and $t= n\theta+\tau/2$, we have
$$\lim_{n\to \infty}\overline{r}_n(s,t) = \lim_{n\to \infty} \frac{K_n(s/n,t/n)}{\sqrt{K_n(s/n,s/n) K_n(t/n,t/n)}} = S(\tau \omega(\theta) )= \frac{\sin (\pi \tau \omega(\theta))}{ \pi \tau \omega(\theta)}$$
Also, by \eqref{eqn:r'}
\begin{align}\label{eqn:r':lim}
\lim_{n \to \infty} \tilde{r}'_n(s,t)=& \lim_{n\to \infty} \frac{1}{v_n(t) \sqrt{K_n(s/n,s/n) K_n(t/n,t/n)}} \left(\frac{K_n^{(0,1)}(s/n,t/n)}{n} - \frac{K_n^{(0,1)}(t/n,t/n)}{n K_n(t/n,t/n)} K_n(s/n,t/n)\right) \nonumber \\
=& \lim_{n\to \infty} \frac{K_n^{(0,1)}(s/n,t/n)}{v_n(t) n\sqrt{K_n(s/n,s/n) K_n(t/n,t/n)}} =\lim_{n\to \infty} \frac{K_n^{(0,1)}(s/n,t/n)}{v_n(t) n {  K_n(\theta, \theta)}}  = \frac{\sqrt 3 S'(\tau \omega(\theta))}{\pi},
\end{align}
where in the last line we used \eqref{eqn:=}, \eqref{eqn:+-} and \eqref{eq:v:lim}.
%(This is in fact natural because $\tilde{r}'$ is the correlation between $\widetilde{T}$ and $\widetilde{T}'$.)

Similarly, by \eqref{eqn:r''} and \eqref{eqn:+-}
\begin{align}\label{eqn:r'':lim}
\lim_{n\to \infty }\tilde{r}''_n(s,t)&= \lim_{n \to \infty} \frac{1}{v_n(s) v_n(t) \sqrt{K_n(s/n,s/n) K_n(t/n,t/n)}}  \Big[\frac{K_n^{(1,1)}(s/n,t/n)}{n^2} \nonumber \\
&- \frac{K_n^{(0,1)}(t/n,t/n)}{n^2 K_n(t/n,t/n)} K_n^{(1,0)}(s/n,t/n) - \frac{K_n^{(0,1)}(s/n,s/n)}{n^2 K_n(s/n,s/n)} K_n^{(0,1)}(s/n,t/n) \nonumber \\
&+ \frac{K_n^{(0,1)}(t/n,t/n)}{n K_n(t/n,t/n)} \frac{K_n^{(0,1)}(s/n,s/n)}{n K_n(s/n,s/n)} K_n(s/n,t/n) \Big] \nonumber \\
&= \lim_{n \to \infty} \frac{K_n^{(1,1)}(s/n,t/n)}{n^2 v_n(s) v_n(t) \sqrt{K_n(s/n,s/n) K_n(t/n,t/n)}} \nonumber \\
& = \frac{1}{ \omega(\theta)^2 \pi^2/3}  [-(\omega(\theta))^{2}S {''}(\tau \omega(\theta))] =-\frac{3}{\pi^2}S {''}(\tau \omega(\theta)).
\end{align}
% \HC{A minus sign was missing; added.}

Putting together, we thus obtain  that for each fixed $\tau$, the following limit exists
\begin{equation}\label{eqn:existence}
\lim_{n\to \infty }  v_n(n\theta+\tau/2) v_n(n\theta-\tau/2)  S_{n,q,\Bd}(n\theta-\tau/2,n\theta+\tau/2) =: f_{q,\Bd}(\theta,\tau).
\end{equation}
Thus, the integrand in \eqref{eq:gn:fn} converges as $n\to\infty$ for fixed $\tau$ and $\theta$. As seen in the proof of Claim \ref{claim:tau:large}, this integrand is uniformly bounded by $B^q$ which depends on $q$ but not on $n,\tau$. Hence there exists a function $h_{q,\Bd}$ such that
$$\lim_{n\to \infty} g_{n,q,\Bd}(\tau) = h_{q,\Bd}(\tau).$$
Since $|g_{n,q,\Bd}(\tau)|\le \frac{B^q}{(|\tau|+1)^{q}}$ by Claim \ref{claim:tau:large}, we also have $|h_{q,\Bd}(\tau)|\le \frac{B^q}{(|\tau|+1)^{q}}$. In particular, it is integrable on $\R$. Therefore, for every $\ep>0$, there exists $T>0$ such that
$$\int_{|\tau|>T} |h_{q,\Bd}(\tau)| d\tau<\ep.$$
To complete the proof of Lemma \ref{lemma:L2:lim}, we write
\begin{eqnarray}
\frac{1}{c_{a,b} n}   \int_{[na,nb]}  \int_{[na,nb]}  v_n(s) v_n(t) S_{n,q,\Bd}(s,t)ds dt &=& \frac{1}{2}\int_{-n(b-a)}^{n(b-a)} g_{n,q,\Bd}(\tau) d\tau\nonumber\\
 &=&  \frac{1}{2}\int_{|\tau| \le T} g_{n,q,\Bd}(\tau) d\tau + \frac{1}{2}\int_{T<|\tau| \le n(b-a)} g_{n,q,\Bd}(\tau) d\tau\nonumber.
\end{eqnarray}
By the Dominated Convergence Theorem,
$$ \left |\int_{|\tau| \le T} g_{n,q,\Bd}(\tau) d\tau - \int_{|\tau| \le T} h_{q,\Bd}(\tau) d\tau\right |\le \ep\text{ for sufficiently large } n.$$
For the remaining term, we have
$$\left|\int_{T<|\tau| \le n(b-a)}g_{n,q,\Bd}(\tau) d\tau \right|  \le \int_{T }^{\infty} \frac{2B^q\,d\tau}{(\tau+1)^{q}} \le \frac{2B^q}{T^{q-1}}\le \ep$$
 by choosing $T$ sufficiently large compared to $q$ and $\ep$. Combining these bounds, we conclude that
 $$\lim_{n \to \infty} \frac{1}{c_{a,b} n}   \int_{[na,nb]}  \int_{[na,nb]}  v_n(s) v_n(t) S_{q,\Bd}(s,t)ds dt= \frac12\int_{\R} h_{q,\Bd}(\tau) d\tau$$
 proving the desired limit in Lemma \ref{lemma:L2:lim}.
\end{proof}

Next we verify Condition \eqref{eqn:finiteness} that $\sigma^2 = \sum_{q=2}^\infty  \sigma_q^2 <\infty$. With $\sigma_{n,q}=\E(I_q^B(g_q^{(n)})^2)$, by Parseval's identity we have $\sum_{q=2}^\infty \sigma_{n,q}^2 =1+o(1)$ because the variance of the LHS term in the decomposition of Theorem \ref{thm:decomposition} is $1+o(1)$.  Hence by Fatou's lemma,
$$\sum_{q=2}^\infty \lim_{n\to \infty} \sigma_{n,q}^2 \le \lim_{n\to \infty} \sum_{q=2}^\infty \sigma_{n,q}^2 =1.$$
The proof is now complete because the LHS above is exactly  $\sum_{q=2}^\infty  \sigma_q^2$.

\section{Verification of Condition \eqref{eqn:tail}}\label{sect:tail}
The proof of this result is a little more involved, where we adopt the ideas from \cite{AADL}. Using the notations from the previous part, see \eqref{eqn:L_2:n}-\eqref{eq:def:S}, we recall that
 \begin{align}\label{eqn:L_2:n'}
&\E(I_q^B(g_q^{(n)})^2)\nonumber\\
&= \frac{1}{c_{a,b} n}   \int_{[na,nb]^2}   \sum_{l,l'=0}^{q/2} b_{q-2l} a_{2l} b_{q-2l'}a_{2l'} \sum_{\Bd \in \BD_{q,2l,2l'}} \frac{(q-2l)! (2l)! (q-2l')! (2l')!} {d_1! d_2! d_3! d_4!}  v_n(s) v_n(t) S_{n,q,\Bd}(s,t)ds dt \nonumber.
%&= \frac{1}{2}\int_{-n(b-a)}^{n(b-a)} \sum_{l,l'=0}^{q/2} b_{q-2l} a_{2l} b_{q-2l'}a_{2l'} \sum_{\Bd \in \BD_{q,2l,2l'}} \frac{(q-2l)! (2l)! (q-2l')! (2l')!} {d_1! d_2! d_3! d_4!} g_{n,\Bd}(\tau) d\tau
\end{align}
%where the second formula was obtained from the first by changing variables and $\tau$ stands for $t-s$ and $g_{n,\Bd}$ is defined in \eqref{eqn:gtau}. In this section it will be more convenient to work with the former form.

We first treat the off-diagonal region. Similarly to the previous section we have
\begin{claim}\label{claim:tau:large'}  There exists a constant  $C$ depending on $a,b, \omega$ such that
	$$\sum_{l,l'=0}^{q/2} |b_{q-2l} a_{2l} b_{q-2l'}a_{2l'}| \sum_{\Bd \in \BD_{q,2l,2l'}} \frac{(q-2l)! (2l)! (q-2l')! (2l')!} {d_1! d_2! d_3! d_4!}  |v_n(s) v_n(t) S_{n,q,\Bd}(s,t)|  \le  \frac{C^q}{(|t-s|+1)^q}$$
	uniformly in $n$ and $\tau=t-s$.
%	As a result,
%$$ |\sum_{l,l'=0}^{q/2} b_{q-2l} a_{2l} b_{q-2l'}a_{2l'} \sum_{\Bd \in \BD_{q,2l,2l'}} \frac{(q-2l)! (2l)! (q-2l')! (2l')!} {d_1! d_2! d_3! d_4!}  g_{n,\Bd}(\tau)| \le  \frac{C^q}{(|\tau|+1)^q}.$$
\end{claim}
% {\color{red} {I removed the part concerning $g_{n, \Bd}$ as it is not needed. The absolute value is moved inside as needed for the next part.}}
\begin{proof} Recall that if $q$ is even then
$$b_{q-2l}= \frac{H_{q-2l}(0)}{\sqrt{2\pi}(q-2l)!} =  (-1)^{q/2-l}\frac{(q-2l-1)!!}{\sqrt{2\pi} (q-2l)!}$$
and
$$a_{2l} =  \sqrt{\frac{2}{\pi}} \frac{1}{2^l l! (2l-1)}.$$
By the convexity of log of the Gamma function, we have $n_1!n_2!\ge (\frac{n_1+n_2}{2}!)^2$ if $n_1$ and $n_2$ have the same parity. Thus, it follows from definition \eqref{eq:def:BD} of $\BD_{q,2l,2l'}$ that
$$d_1! d_2! d_3! d_4! \ge (q/2-l)! (q/2-l')! l! l'!.$$
Hence we have the following rather generous estimate
$$|b_{q-2l} a_{2l} b_{q-2l'} a_{2l'}| \frac{(q-2l)! (2l)! (q-2l')! (2l')!}{d_1!d_2!d_3!d_4!}  \le \frac{(q-2l)!! (2l)!! (q-2l')!! (2l')!!}{ (q/2-l)! (q/2-l')! l! l'!} \le 4^q.$$
The claim then follows by using \eqref{eq:def:S} and Lemma \ref{lemma:bounds} as in Claim \ref{claim:tau:large}.
\end{proof}
It thus follows that for $T_0$ sufficiently large,
 \begin{equation}\label{eqn:off-diag}
\sum_{q\ge Q}  \frac{1}{c_{a,b} n}   \int_{\substack{|s-t| \ge T_0\\ s, t\in [na, nb]}}    \sum_{l,l'=0}^{q/2} |b_{q-2l}\cdots a_{2l'}| \sum_{\Bd \in \BD_{q,2l,2l'}} \frac{(q-2l)!\cdots(2l')!} {d_1! d_2! d_3! d_4!}  |v_n(s) v_n(t) S_{n,q,\Bd}(s,t) |ds dt  \le (C'/T_0)^{Q-1}.
\end{equation}
%\begin{equation}\label{eqn:off-diag'}
%\sum_{q\ge Q}   \int_{|\tau| > T_0}  |\sum_{l,l'=0}^{q/2} b_{q-2l} a_{2l} b_{q-2l'}a_{2l'} \sum_{\Bd \in \BD_{q,2l,2l'}} \frac{(q-2l)! (2l)! (q-2l')! (2l')!} {d_1! d_2! d_3! d_4!}  g_{n,\Bd}(\tau)| d\tau \le (C'/T_0)^{Q-1}
%\end{equation}
which converges to $0$ as $Q\to \infty$.
%the to show that as $Q \to \infty$,
%$$\sum_{q\ge Q}  \int_{|\tau| \le T_0} \sum_{l,l'=0}^{q/2} b_{q-2l} a_{2l} b_{q-2l'}a_{2l'} \sum_{\Bd \in \BD_{q,2l,2l'}} \frac{(q-2l)! (2l)! (q-2l')! (2l')!} {d_1! d_2! d_3! d_4!} g_{n,\Bd}(\tau) d\tau \to 0.$$
\subsection{Diagonal term} Hence for our main result, it suffices to deal with the diagonal region $|\tau| \le T_0$ (or $|s-t| \le T_0$), more precisely we will need to show
$$\lim_{Q\to \infty} \limsup_{n\to \infty}\sum_{q\ge Q}  \frac{1}{c_{a,b} n}   \int_{\mathcal I}  \sum_{l,l'=0}^{q/2} b_{q-2l}\cdots a_{2l'} \sum_{\Bd \in \BD_{q,2l,2l'}} \frac{(q-2l)! \cdots (2l')!} {d_1! d_2! d_3! d_4!}  v_n(s) v_n(t) S_{n,q,\Bd}(s,t) ds dt    = 0$$
for some $\mathcal I$ containing the region $\{(s,t) \in (na,nb)^{2}: |s-t| \le T_0\}$.
We divide the interval $(na,nb)$ into $\Theta(n)$ sub-intervals of length $T_0$
$$I_{i}= (n\theta_i, n \theta_i +T_0), \theta_i = a+  i T_0/n, 0\le i \le n(b-a)/T_0.$$
%Because of \eqref{eqn:off-diag}, it suffices to show
%$$\lim_{Q\to \infty} \limsup_{n\to \infty} \frac{1}{c_{a,b}n} \sum_{|i-j|\le 1} \sum_{q\ge Q}  \int_{(s,t)\in I_i \times I_j} \sum_{l,l'=0}^{q/2} b_{q-2l}\cdots a_{2l'} \sum_{\Bd \in \BD_{q,2l,2l'}} \frac{(q-2l)! \cdots (2l')!} {d_1! d_2! d_3! d_4!}   v_n(s) v_n(t) S_{n,q,\Bd}(s,t) ds dt = 0.$$
We let $$\mathcal I = \cup_{|j-i|\le 1} I_i\times I_j.$$
Since $(I_i \times I_{i+1}) \cup (I_{i+1} \times I_i) = ((I_i \cup I_{i+1}) \times (I_i \cup I_{i+1})) \bs (I_i \times I_i) \cup (I_{i+1} \times I_{i+1})$, by the triangle inequality and by replacing $I_i$ by $I_{i}\cup I_{i+1}$ if needed and similarly for $I_i\times I_{i-1}$, it suffices to work with the simplified sum $\frac{1}{c_{a,b}n} \sum_{i} \sum_{q\ge Q}  \int_{s,t\in I_i \times I_i}$. To this end,
the  key observation is that, by \eqref{eqn:decomposition'}, for each $i$ the integral sum $\sum_{q\ge Q}  \int_{s,t\in I_i \times I_i}$ corresponds to the tail of the variance of $N_n(I_i)$ (i.e. the number of roots of $\overline{T}_n$ over $I_i$), more precisely
\begin{align*}
& \sum_{q\ge Q}  \int_{(s,t)\in I_i \times I_i} \sum_{l,l'=0}^{q/2} b_{q-2l} a_{2l} b_{q-2l'}a_{2l'} \sum_{\Bd \in \BD_{q,2l,2l'}} \frac{(q-2l)! (2l)! (q-2l')! (2l')!} {d_1! d_2! d_3! d_4!}   v_n(s) v_n(t) S_{n,q,\Bd}(s,t) ds dt\\
&= \sum_{q\ge Q} \Var \int_{I_i} f_q(\overline{T}_n(s) , \CT_n'(s)) v_n(s)ds.
\end{align*}

%  {\text{\color{red} some factor missing?} \HC{No, (20) was stated for non-normalizing.}

Hence, we would like to show that (where we replace $c_{ab}n$ by $(b-a)n/T_0$ to make the expression more natural)
\begin{equation}\label{eqn:hQ}
\lim_{Q \to \infty} \limsup_{n\to \infty}\frac{1}{(b-a)n/T_0} \sum_i \sum_{q \ge Q} \Var\Big(\int_{I_i} f_q(\overline{T}_n(s) , \CT_n'(s)) v_n(s)ds\Big)=0.
\end{equation}

% {\color{red} where does the $\frac{1}{(b-a)/T_0} $ comes from?}

Now for each $\theta \in (a,b)$, with $I_{\theta} := (n\theta, n\theta +T_0)$, it follows from Section \ref{sec:proof:bounds} and from Section \ref{sect:single} that within $I_\theta$, the process $\overline{T}_n$ converges to the (stationary) gaussian process $\overline{T}_{\infty, \theta}$ of zero mean and covariance
\begin{equation}\label{eqn:cov:T}
r(s,t) = S((s-t) \omega(\theta)) = \frac{\sin(\pi (s-t) \omega(\theta) )}{\pi (s-t) \omega(\theta)}.
\end{equation}
More precisely, we prove the following.
%  {\color{red} I restated the result; only first order derivative involved, not second order, right?}

%\begin{lemma}\label{lemma:local:single} As $n\to \infty$, the rescaled processes $\overline{T}_n$ and its first and second order derivatives over $I_\theta$ converge (in the sense of distribution) \HC{unifomly in $\theta$} to the above gaussian process  $\overline{T}_{\infty, \theta}$  and its respective derivatives over $(0,T_0)$.
%\end{lemma}

\begin{lemma}\label{lemma:local:single} Uniformly in $\theta\in (a, b)$, the following convergence holds. As $n\to \infty$,  uniformly in $(s, t)\in \mathcal I_\theta\times \mathcal I_\theta$, the second moments and covariances of the rescaled processes $\overline{T}_n$ and its first order derivatives converge to those of the Gaussian process  $\overline{T}_{\infty, \theta}$.
\end{lemma}

\begin{proof} For convenience let $X(t)$ be the process $\overline{T}_{\infty, \theta}(t)$, which has covariance $\E(X(s)X(t))$ as in \eqref{eqn:cov:T}. Both $X(t)$ and $\overline{T}_n(t)$ have unit variance. We have 
$$\E |X'(t)|^2 = \lim_{s\to t} \E (\frac{X_s-X_t}{s-t})^2 =  \lim_{s\to t}  \frac{1}{(s-t)^2}(2 -2 r(s,t))   = \omega^2(\theta) \pi^2 /3 = \lim_{n\to\infty} \E |\overline{T}'_n(t)|^{2},$$
where in the last equality we used \eqref{def:vn} and \eqref{eq:v:lim}. 

Furthermore
$$\E X(s) X'(t) = \E \lim_{t'\to t} X(s)\frac{X(t')-X(t)}{t'-t}  =  \lim_{t'\to t} \frac{1}{t'-t} (r(s,t') - r(s,t))=\frac{\partial r(s,t)}{\partial t} = -\omega(\theta) S'((s-t)\omega(\theta)).$$
And so, 
$$ \frac{\E X(s) X'(t)}{\sqrt{\E |X'(t)|^2}} = \lim _{n\to \infty} \tilde{r}_n'(s,t) = \lim _{n\to \infty} \frac{\E \overline{T}_n(s) \overline{T}'_n(t)}{\sqrt{\E |\overline{T}'_n(t)|^2}}$$
where we used \eqref{def:tilde:r'} and \eqref{eqn:r':lim} (and noted that $\tau=t-s$ in the latter equation).

%\HC{ Fixed the minus sign mismatched; changed from s to t}

Lastly,
\begin{align*}\E X'(s) X'(t) &= \E \lim_{s'\to s, t'\to t} \frac{X(s')-X(s)}{s'-s} \frac{X(t')-X(t)}{t'-t}\\
&= \lim_{s'\to s, t'\to t} \frac{1}{s'-s} \frac{1}{t'-t}  (r(s',t') - r(s',t) - r(s,t')+ r(s,t)) \\
& = -(\omega(\theta))^2 S''((s-t)\omega(\theta)).
\end{align*}
After normalizing by $\sqrt{\E |X'(s)|^2}$ and $\sqrt{\E |X'(t)|^2}$, we again obtain the same as $\lim_{n\to \infty} \tilde{r}_n''(s,t)$ in \eqref{eqn:r'':lim} which is the corresponding identity for $\overline{T}_n$ by \eqref{def:tilde:r''}.
\end{proof}
We continue with several other pleasant properties.
\begin{lemma}\label{lemma:local:var} The following holds uniformly in $\theta$
\begin{enumerate}[(i)]
\item For each $q\ge 2$ we have
$$\lim_{n\to \infty} \Var\Big(\int_{I_\theta} f_q(\overline{T}_n(s) , \CT_n'(s)) v_n(s)ds\Big) = \Var\Big(\int_{(0,T_0)} f_q(\overline{T}_{\infty, \theta}(s) , \CT_{\infty,\theta}'(s)) v_\infty (s)ds\Big).$$
\vskip .1in
\item Also,
$$\lim_{n\to \infty} \Var(N_n({I_\theta})) = \Var (N_{\infty,\theta} ((0,T_0))),$$
where $N_n({I_\theta})$ and $N_{\infty,\theta}$ are the number of roots with respect to the processes $\overline{T}_n$ and $T_{\infty,\theta}$ over the intervals ${I_\theta}$ and $(0,T_0)$ respectively.
\vskip .1in
\item There exists $C_0$ not depending of $n$ and $\theta$ so that
 $$\Var(N_n({I_\theta})) \le C_0.$$
 \end{enumerate}
 \end{lemma}

\begin{proof} The proof of (i) is similar to our verification of Condition \eqref{eqn:single} in Section \ref{sect:single}, where we can use the fact that $\Var(f_q(\overline{T}_n(s) , \CT_n'(s)))$ is a polynomial of $\tilde{r}_n(.), \tilde{r}_n'(.), \tilde{r}_n''(.)$, and that these correlations converge to their corresponding parts of $\CT_{\infty,\theta}$ uniformly in $\theta$ owing to Lemma \ref{lemma:local:single}.

For (ii) we first use Lemma \ref{lemma:2.2} to obtain a formula for $\Var(N_n(I))$, and then use (a) of Lemma \ref{lemma:2.4} and Lemma \ref{lemma:local:single} to compare the intensities $\rho_2$ and $\rho_1$ with their corresponding parts in the Kac-Rice's formula of $\Var(N_{\infty,\theta} ((0,T_0)))$. We leave the details to the reader. Finally, for (iii) we just use \eqref{eqn:rho1} and (b) of  Lemma \ref{lemma:2.4}.
\end{proof}

We now conclude the section.

\begin{proof}[Proof of \eqref{eqn:hQ} (and hence of Condition \eqref{eqn:tail})] Let $\eps>0$. Let $Q$ be chosen later. By (i) of Lemma \ref{lemma:local:var} the following holds uniformly in $\theta$
\begin{align*}
\lim_{n\to \infty} \sum_{q< Q}  \Var\Big(\int_{I_{\theta }} f_q(\overline{T}_n(s) , \CT_n'(s)) v_n(s)ds\Big) &= \sum_{q < Q}  \Var\Big(\int_{(0,T_0)} f_q(\overline{T}_{\infty, \theta}(s) , \CT_{\infty,\theta}'(s)) v_\infty (s)ds\Big).
\end{align*}
It thus follows from (ii) and (iii) of Lemma \ref{lemma:local:var} that for sufficiently large $n$,  uniformly in $\theta \in [a,b],$
$$\Big| \sum_{q \ge Q}  \Var\Big(\int_{I_{\theta}} f_q(\overline{T}_n(s) , \CT_n'(s)) v_n(s)ds\Big) - \sum_{q \ge Q}  \Var\Big(\int_{(0,T_0)} f_q(\overline{T}_{\infty, \theta}(s) , \CT_{\infty,\theta}'(s)) v_\infty (s)ds\Big)\Big| \le \eps.$$
Hence by the triangle inequality, for sufficiently large $n$
$$\Big |\frac{1}{(b-a)n/T_0} \sum_i \sum_{q \ge Q} \Var \Big(\int_{I_{i}} f_q(\overline{T}_n(s) , \CT_n'(s)) v_n(s)ds\Big)$$
$$- \frac{1}{(b-a)n/T_0} \sum_i   \sum_{q \ge Q} \Var \Big(\int_{(-T_0/2,T_0/2)} f_q(\overline{T}_{\infty, \theta_i}(s) , \CT_{\infty,\theta_i}'(s)) v_\infty (s)ds\Big)\Big| \le \eps.$$
Now we consider the second sum involving the limiting process $\overline{T}_{\infty}$. As  $\sum_{q < Q} \Var (\int_{(0,T_0)} f_q(\overline{T}_{\infty, \theta}(s) , \CT_{\infty,\theta}'(s)) v_\infty (s)ds)$ and  $\Var (N_{\infty,\theta} ((0,T_0)))$ are both continuous and uniformly bounded for $\theta \in [a,b]$, the tail function $\sum_{q \ge Q} \Var (\int_{(0,T_0)} ... ds)$ is also continuous and uniformly bounded, and hence the Riemann sum converges to its integral,
   $$\lim_{n\to \infty} \frac{1}{(b-a)n/T_0} \sum_i   \sum_{q \ge Q} \Var \left(\int_{(0,T_0)} f_q(\overline{T}_{\infty, \theta_i}(s) , \CT_{\infty,\theta_i}'(s)) v_\infty (s)ds \right)$$
  $$=  \int_{a}^b \sum_{q \ge Q} \Var \left(\int_{(0,T_0)} f_q(\overline{T}_{\infty, \theta}(s) , \CT_{\infty,\theta}'(s)) v_\infty (s)ds \right) d\theta.$$
Passing back to the original sum, for sufficiently large $n$ we have
 $$\frac{1}{(b-a)n/T_0} \sum_i \sum_{q \ge Q} \Var \left(\int_{I_{i}} f_q(\overline{T}_n(s) , \CT_n'(s)) v_n(s)ds\right)$$
 $$ \le  \int_{a}^b \sum_{q \ge Q} \Var \left(\int_{(0,T_0)} f_q(\overline{T}_{\infty, \theta_i}(s) , \CT_{\infty,\theta_i}'(s)) v_\infty (s)ds\right) d\theta + 2\eps.$$
To this end,  again by  (ii) and (iii) of Lemma \ref{lemma:local:var} and by Fubini,
 $$\sum_{q} \int_{a}^b \Var \left(\int_{(0,T_0)} f_q(\overline{T}_{\infty, \theta_i}(s) , \CT_{\infty,\theta_i}'(s)) v_\infty (s)ds\right) d\theta <\infty.$$
Hence there exists $Q$ such that $\sum_{q \ge Q} \int_{a}^b \Var (\int_{(0,T_0)} f_q(\overline{T}_{\infty, \theta_i}(s) , \CT_{\infty,\theta_i}'(s)) v_\infty (s)ds) d\theta <\eps$, which in turn (again via Fubini) implies that
$$\int_{a}^b\sum_{q \ge Q}  \Var \left(\int_{(0,T_0)} f_q(\overline{T}_{\infty, \theta_i}(s) , \CT_{\infty,\theta_i}'(s)) v_\infty (s)ds\right) d\theta <\eps.$$
We have thus shown that for any given $\eps$, there exists a large $Q$ such that for all sufficiently large $n$ we have
 $$\frac{1}{(b-a)n/T_0} \sum_i \sum_{q \ge Q} \Var \left(\int_{I_{i}} f_q(\overline{T}_n(s) , \CT_n'(s)) v_n(s)ds\right) \le 3\eps.$$
The proof of  \eqref{eqn:hQ} is then complete by sending $\eps \to 0$.
 \end{proof}

\section{Proof of Theorem \ref{thm:CLT}}
 By Theorem \ref{thm:gau} and Proposition \ref{prop:cov:zero}, for each fixed $Q$ we have that $\sum_{q=2}^Q I_q^B(g_q^{n})$ converges in distribution to $\BN(0,\sum_{q=2}^Q \sigma_q^2)$. On the other hand, recall that by Theorem \ref{thm:decomposition},
$$\frac{N([na,nb]) - \E N}{\sqrt{c_{a,b}n}} = \sum_{q=2}^\infty I_q^B(g_q^{(n)})$$
and
$$\lim_{Q \to \infty }\sum_{q=2}^Q \sigma_q^2=1.$$
It thus follows that $\frac{N([na,nb]) - \E N}{\sqrt{c_{a,b}n}}$ converges to $\BN(0,1)$ in distribution, completing the proof.

\section{Proof of Theorem ~\ref{thm:decomposition}}\label{sect:decomp}

We will need the following analog of \cite[Lemma 2]{KL} on Hermite expansion of Kac-Rice formula.

\begin{lemma}\label{l.1} The following holds
	\begin{enumerate}[(i)]
		%\item[(i)] $\mathbb E|N([na,nb])|^2<\infty$.
		\item  Define
		$$N^\eta([na,nb]) :=\int_{na}^{nb}  \phi_\eta(\overline {T}_n(s))|\CT'_n(s)|v_n(s)ds$$
		where $\phi_\eta$ is the density of the $N(0,\eta)$ distribution. Then $N^\eta([na,nb])$ converges almost surely and in $L^2$ to $N([na,nb])$, and $\mathbb E|N^\eta([na,nb])|^2 \to \mathbb E |N([na,nb])|^2$, as $\eta\to 0.$
		\vskip .05in
		\item The random variable $N^\eta([na,nb])$ has the chaotic expansion (in $L^2$)
		$$N^\eta([na,nb])=\sum_{q=0}^\infty \sum_{\ell=0}^{\lfl q/2 \rfl} b^\eta_{q-2\ell}a_{2\ell}\int_{na}^{nb}H_{q-2\ell}(\overline{T}_n(s))H_{2\ell}(\CT'_n(s))v_n(s)ds,$$
		where $b^\eta_{k}$ are the Hermite coefficients of $\phi_{\eta}$,  $b^\eta_k = \frac{1}{\sqrt{2\pi} k!}\int_{-\infty}^\infty \phi_\eta(t)H_k(t) e^{-t^2/2}dt$.
	\end{enumerate}
\end{lemma}

% \HC{Replaced $b^\eta_k = \int_{-\infty}^\infty \phi_\eta(t)H_k(t) e^{-t^2/2}dt.$}

Assuming this lemma for the moment, we will conclude our main result.
\proof[Proof of Theorem~\ref{thm:decomposition}] Since $H_k$ is odd if $k$ is odd, it follows immediately that $b^\eta_k=0=b_k$ for $k$ odd. For $k$ even, we have
$$\lim_{\eta\to 0}b_k^\eta= \frac{H_k(0)}{\sqrt{2\pi}k!} =b_k.$$

Given any $Q$, using Fatou's lemma and Lemma \ref{l.1} we have
\begin{eqnarray*}
 && \sum_{q=0}^Q \mathbb E\Big(\Big[\sum_{\ell=0}^{\lfl q/2 \rfl} b_{q-2\ell}a_{2\ell}\int_{na}^{nb}H_{q-2\ell}(\overline{T}_n(s))H_{2\ell}(\CT'_n(s))v_n(s)ds\Big]^2\Big) \\
 &\le& \liminf_{\eta\to 0}\sum_{q=0}^Q \mathbb E\Big(\Big[\sum_{\ell=0}^{\lfl q/2 \rfl} b^\eta_{q-2\ell}a_{2\ell}\int_{na}^{nb}H_{q-2\ell}(\overline{T}_n(s))H_{2\ell}(\CT'_n(s))v_n(s)ds\Big]^2\Big) \\
 &\le& \liminf_{\eta\to 0}\sum_{q=0}^\infty \mathbb E\Big(\Big[\sum_{\ell=0}^{\lfl q/2 \rfl} b^\eta_{q-2\ell}a_{2\ell}\int_{na}^{nb}H_{q-2\ell}(\overline{T}_n(s))H_{2\ell}(\CT'_n(s))v_n(s)ds\Big]^2\Big) \\
 &=& \liminf_{\eta\to 0}\mathbb E|N^\eta([na,nb])|^2 = \mathbb E |N([na,nb])|^2 <\infty,
\end{eqnarray*}
where we used the orthogonality of the chaos (see the discussion in the Chaining section after the statement of Theorem~\ref{thm:decomposition}), and in the last inequality we used $N([na,nb])\le n$. %See for instance \cite[Corollary 8.14]{PT}

It follows that the following expression converges to a limit $L$ in $L^2$, as $\eta\to 0$:
$$\sum_{q=0}^\infty \sum_{\ell=0}^{\lfl q/2 \rfl} b_{q-2\ell}a_{2\ell}\int_{na}^{nb}H_{q-2\ell}(\overline{T}_n(s))H_{2\ell}(\CT'_n(s))v_n(s)ds.$$
Our first goal is to show that, almost surely, this limit $L$ is exactly $N([na,nb])$. By the triangle inequality,
$$\|N([na,nb])-L\|_2 \le \|N(na,nb)-N^\eta([na,nb])\|_2 + \|N^\eta([na,nb])-L\|_2,$$
and by Lemma~\ref{l.1} part (i) the first term on the right hand side converges to $0$ as $\eta\to 0$. On the other hand, letting
$$J_q:=\int_{na}^{nb}H_{q-2\ell}(\overline{T}_n(s))H_{2\ell}(\CT'_n(s))v_n(s)ds,$$
using the facts that $\overline{T}_n(s)$ and $\CT'_n(s)$ are standard Gaussian, and that $\int_{na}^{nb}v_n(s)ds=N([na,nb])<\infty$, it can be seen that $J_q\in L^2$. We then  use part (ii) of Lemma~\ref{l.1} to write
$$\|N^\eta([na,nb])-L\|_2 =\left\|\sum_{q=0}^\infty \sum_{\ell=0}^{\lfl q/2\rfl}(b^\eta_{q-2\ell}-b_{q-2\ell})a_{2\ell}J_q \right\|_2$$
$$\le \left\|\sum_{q=0}^Q \sum_{\ell=0}^{\lfl q/2\rfl}(b^\eta_{q-2\ell}-b_{q-2\ell})a_{2\ell}J_q\right\|_2 + \left\|\sum_{q=Q+1}^\infty \sum_{\ell=0}^{\lfl q/2\rfl}b^\eta_{q-2\ell}a_{2\ell}J_q\right\|_2 + \left\|\sum_{q=Q+1}^\infty \sum_{\ell=0}^{\lfl q/2\rfl}b_{q-2\ell} a_{2\ell}J_q\right\|_2$$
for any fixed $Q$. Thus if $Q$ is fixed then for $\eta\to 0$ the first term converges to $0$ (as $b^\eta_k\to b_k$ for fixed $k$). If $\mathcal P_{>Q}$ denotes the projection onto the (random) subspace $\oplus_{j>Q} \CH_{j}$, then the second term is exactly $\|\mathcal P_{>Q}(N^\eta([na,nb]))\|$, and thanks to boundedness of $\mathcal P_{>Q}$ and thanks to Lemma~\ref{l.1} this term converges to $\|\mathcal P_{>Q}(N([na,nb]))\|_2$. The third term is exactly $\|\mathcal P_{>Q}(L)\|_2$. Consequently
$$\limsup_{\eta\to 0}\|N^\eta([na,nb])-L\|_2 \le\|\mathcal P_{>Q}(N([na,nb]))\|_2 + \|\mathcal P_{>Q}L\|_2.$$
Since both $N([na,nb])$ and $L$ belong to $L^2$, by taking $Q$ large enough, we obtain that
$$0\le \limsup_{\eta\to 0}\|N^\eta([na,nb])-L\|_2 \le \epsilon$$
for any $\epsilon>0$,  thus the left hand side limit must be zero. Collecting estimates, we arrive at the conclusion that $\|N([na,nb])-L\|_2=0$, and hence $L=N([na,nb])$.

Thus, we just showed that the following chaotic expansion holds (in $L^2$)
\begin{eqnarray}
 N([na,nb]) &=&\sum_{q=0}^\infty \sum_{\ell=0}^{\lfl q/2 \rfl} b_{q-2\ell}a_{2\ell}\int_{na}^{nb}H_{q-2\ell}(\overline{T}_n(s))H_{2\ell}(\CT'_n(s))v_n(s)ds\nonumber\\
 &\equiv& \sum_{q=0}^\infty \sqrt{c_{a,b} n} \ I_q^{\overline{T}_n}([na,nb]).\nonumber
\end{eqnarray}

Now, observe that $I_1^{\overline{T}_n}=0$ (as $b_1=0$). On the other hand, as $b_0=c_0=H_0=1$, we have
$$I_0^{\overline{T}_n}=\frac{1}{\sqrt{c_{a,b} n}} \int_{na}^{nb}v_n(s)ds=\frac{1}{\sqrt{c_{a,b} n}}\mathbb E N([na,nb]),$$
where in the last equality we used the Kac-Rice formula (in Edelman-Kostlan's formulation).

Consequently the following expansion holds in $L^2$
$$\frac{N([na,nb])-\mathbb EN([na,nb])}{\sqrt {c_{a,b} n}} = \sum_{q=2}^\infty I_q^{\overline{T}_n}([na,nb]),$$
as desired.
\endproof

\proof[Proof of Lemma~\ref{l.1}]
(i)  We start with almost surely convergence. Let $N([na,nb])(u)$ denote the number of crossings of $\overline{T}_n$ of level $u$.

We first use the area formula \cite{Federer1966} to write
$$N^\eta([na,nb])=\int_{-\infty}^\infty N([na,nb])(u)\phi_\eta(u)du.$$

We now show that $N([na,nb])(u)$ is locally continuous at $u=0$. Since $N([na,nb])(0)$ is  bounded above by $n$, it suffices to show that
\begin{lemma}\label{l.dr}Almost surely $\overline{T}_n$ has no double root in $[na,nb]$.
\end{lemma}
We will defer the proof of this result to the end of the section. It follows that given any $\epsilon>0$ there is some $\delta>0$ such that
$$|N([na,nb])(u)-N([na,nb])(0)|<\epsilon$$
for $|u|<\delta$. We choose $\eta \ll \delta$ and write
\begin{eqnarray}
 &&N^\eta([na,nb]) - N([na,nb]) =\int_{|u|>\delta}^\infty N([na,nb])(u) \phi_\eta(u)du  \nonumber\\
 &&\quad  - \int_{|u|>\delta}^\infty  N([na,nb])\phi_\eta(u)du + \int_{|u|\le \delta}^\infty \Big(N([na,nb])(u)-N([na,nb])(0)\Big)\phi_\eta(u)du.\nonumber
\end{eqnarray}

Note that $N([na,nb])(0)\le n$ and $1_{|u|>\delta}\phi_\eta(u)$ decreases to $0$ as $\eta\to 0$, thus by dominated convergence the second term converges to $0$ as $\eta\to 0$.

Similarly $N([na,nb])(u) \phi_\eta(u)1_{|u|>\delta}$ decreases to $0$ as $\eta\to 0$, and by dominated convergence  the first term also converges to $0$ as $\eta\to 0$.

The last term is bounded above by
$$\int_{|u|\le \delta} \epsilon \phi_\eta(u)du < \epsilon \int_{-\infty}^\infty \phi_\eta(u)du=\epsilon.$$

Consequently,
$$\limsup_{\eta\to 0} |N^\eta([na,nb]) - N([na,nb])|\le \epsilon$$
for any $\epsilon>0$. This proves the almost sure convergence of $N^\eta([na,nb])$ to $N([na,nb])$ when $\eta\to 0$.

We now prove that $\lim_{\eta \to 0} \mathbb E |N^\eta([na,nb])|^2 = \mathbb E |N([na,nb])|^2$.

By Fatou's lemma, we first have
$$\mathbb E |N([na,nb])|^2 \le  \liminf_{\eta\to 0} \mathbb E |N^\eta([na,nb])|^2.$$

Now, let $N'([na,nb])$ denote the number of zeros of $\CT'_n$ (equivalently of $\overline{T}'_n$) in the same interval. We first note the following.
\begin{lemma}\label{lemma:N_X}
We have  $N'([na,nb])$ is square integrable.
\end{lemma}
We will prove this result later. To proceed, it is clear that   $\sup_u N([na, nb])(u) \le 1+ N'([na, nb])$. As $N([na,nb])(u) \to N([na,nb])$ when $u\to 0$, it follows from an application of the Dominated Convergence Theorem that
$$\lim_{u\to 0}\mathbb E|N([na,nb])(u)|^2=\mathbb E |N([na,nb])|^2.$$

The above argument also reveals that $\mathbb E|N([na,nb])(u)|^2$ is uniformly bounded by $\mathbb E|1+N'([na,nb])|^2<\infty$.

Now, using the area formula we have
$$\mathbb E |N^\eta([na,nb])|^2 = \mathbb E \left|\int_{-\infty}^\infty N([na,nb])(u)\phi_\eta(u)du \right|^2 \le \mathbb E \int_{-\infty}^\infty |N([na,nb])(u)|^2\phi_\eta(u)du $$
$$= \int_{-\infty}^\infty \mathbb E |N([na,nb])(u)|^2\phi_\eta(u)du < \infty.$$
Since $\mathbb E |N([na,nb])(u)|^2$ is continuous at $u=0$ and uniformly bounded, and since $\phi_\eta\to \delta_0$ in distribution,
$$\lim_{\eta\to 0} \int_{-\infty}^\infty \mathbb E |N([na,nb])(u)|^2\phi_\eta(u)du=\mathbb E |N([na,nb])(0)|^2=\mathbb E |N([na,nb])|^2.$$
We obtain
$$\limsup_{\eta\to 0}\mathbb E |N^\eta([na,nb])|^2 \le \mathbb E |N([na,nb])|^2.$$
Collecting estimates, we obtain the desired $L^2$ equality
$$\lim_{\eta\to 0}\mathbb E |N^\eta([na,nb])|^2=\mathbb E |N([na,nb])|^2.$$
From this and the almost sure convergence, it follows via an application of Fatou's lemma that
$$2\mathbb E |N([na,nb])|^2 \le \liminf_{\eta\to 0} \left(\mathbb E |N([na,nb])|^2 +|N^\eta([na,nb])|^2 - |N([na,nb])-N^\eta([na,nb])|^2 \right),$$
which equals
$$ 2\mathbb E |N([na,nb])|^2 - \limsup_{\eta\to 0} \mathbb E |N([na,nb])-N^\eta([na,nb])|^2.$$
It follows that
$$\lim_{\eta\to 0}\mathbb E |N([na,nb])-N^\eta([na,nb])|^2 = 0,$$
i.e., $N^\eta([na,nb]) \to N([na,nb])$ in $L^2$.

(ii) By direct computation (and definition of $b^\eta_k$) we have the Hermite expansions  (see also \cite{ADL})
$$|\xi|=\sum_{k=0}^\infty a_k H_k(\xi),$$
$$\phi_\eta(\xi)=\sum_{k=0}^\infty b_k^\eta H_k(\xi).$$

Denote $X_s=\overline{T}_n(s)$ and $Y_s=\CT'_n(s)$, both are mean zero Gaussian processes with variance one. Consider
$$\xi^L(s):=\sum_{\ell=0}^L a_{2\ell} H_{2\ell}(Y_s),$$
$$\gamma^Q(s):=\sum_{k=0}^Q b_k^\eta H_k(X_s),$$
$$N^{Q,L}:=\int_{na}^{nb}\gamma^Q(s)\xi^L(s)v_n(s)ds.$$
Note that $\xi^L(s)$ converges in $L^2$ to $|Y_s|$ and $\gamma^Q$ converges in $L^2$ to $\phi_\eta(X_s)$. By convexity and the triangle inequality, for any $Q,L,Q',L'$
$$|N^{Q',L'}-N^{Q,L}|  \le \int_{na}^{nb}v_n(s)|(\gamma^{Q'}(s)-\gamma^Q(s))\xi^{L'}(s)|ds+\int_{na}^{nb}v_n(s)|\gamma^Q(s)(\xi^{L'}(s)-\xi^L(s))|ds.$$
So,
$$(\mathbb E|N^{Q',L'}-N^{Q,L}|^2)^{1/2}  \le \int_{na}^{nb}v_n(s)\big(\mathbb E|\gamma^{Q'}(s)-\gamma^Q(s)|^2 |\xi^{L'}(s)|^2\big)^{1/2}ds$$
$$+\int_{na}^{nb}v_n(s)\big(\mathbb E |\gamma^Q(s)|^2(\xi^{L'}(s)-\xi^L(s))^2\big)^{1/2}ds.$$
Using independence of $X_s$ and $Y_s$, the last right hand side of the above equation is the same as
$$\int_{na}^{nb}v_n(s)\big(\mathbb E|\gamma^{Q'}(s)-\gamma^Q(s)|^2\big)^{1/2} \big(\mathbb E |\xi^{L'}(s)|^2\big)^{1/2}ds +\int_{na}^{nb}v_n(s)\big(\mathbb E |\gamma^Q(s)|^2\big)^{1/2}\big(\mathbb E(\xi^{L'}(s)-\xi^L(s))^2\big)^{1/2}ds$$
$$\le \int_{na}^{nb}v_n(s)\big(\mathbb E|\gamma^{Q'}(s)-\gamma^Q(s)|^2\big)^{1/2} (\mathbb E |X_s|^2)^{1/2}ds +\int_{na}^{nb}v_n(s)\big(\mathbb E |Y_s|^2\big)^{1/2}\big(\mathbb E(\xi^{L'}(s)-\xi^L(s))^2\big)^{1/2}ds.$$
Note that $\mathbb E |X_s|^2=\mathbb E|Y_s|^2=1$. Using orthogonality,
$$\mathbb E|\gamma^{Q'} - \gamma^Q|^2 = \sum_{k \text{ between Q,Q'}}|b_k^\eta|^2k! \to 0 \mbox{ as } \min(Q,Q')\to \infty$$
 because $\sum_{k\ge 0}|b_k^\eta|^2 k!=\mathbb E|X_s|^2=1$.

  % \HC{Need to verify this identity.}

 Similarly, $\mathbb E|\xi^{L'} - \xi^L|^2\to 0$ as $ \min(L,L')\to \infty$.

 Since $\int_{na}^{nb}v_n(s)ds=\mathbb E N([na,nb])<\infty$, by Dominated Convergence Theorem we obtain
\begin{equation}\label{eqn:QLQ'L'}
\lim_{\min(Q,L,Q',L')\to \infty} (\mathbb E|N^{Q',L'}-N^{Q,L}|^2)^{1/2} = 0.
\end{equation}
Note that the above argument works even if $Q'=L'=\infty$, thus we actually proved that
\begin{equation}\label{eqn:QLinfty}
\lim_{\min(Q,L)\to \infty} \left(\mathbb E|N^{Q,L} - \int_{na}^{nb}\phi_\eta(X_s)|Y_s|v_n(s)ds|^2\right)^{1/2} = 0.
\end{equation}
To proceed, we observe that for any $(k,\ell)\ne (k',\ell')$ and any $s$ the following holds:
\begin{equation}\label{eqn:orth}
\mathbb E \Big(H_k(X_s)H_{\ell}(Y_s)H_{k'}(X_s)H_{\ell'}(Y_s)\Big)=0.
\end{equation}
(Indeed, without loss of generality assume $k\ne k'$. In this case, by the independence of $X_s$ and $Y_s$, we notice that $\mathbb E(H_k(X_s)H_{k'}(X_s))=0$ as $X_s$ is standard Gaussian.)

Thus by \eqref{eqn:orth}
$$\Big(\mathbb E\Big|\sum_{\ell+k \le Q} b_k^\eta a_{\ell}\int_{na}^{nb}H_{k}(X_s)H_{\ell}(Y_s)v_n(s)ds - N^{Q,Q}\Big|^2\Big)^{1/2}$$
$$=\Big(\mathbb E \Big|\sum_{q=Q+1}^{2Q} \sum_{\ell=q-Q}^{Q} b_{q-\ell}^\eta a_{\ell}\int_{na}^{nb}H_{k}(X_s)H_{\ell}(Y_s)v_n(s)ds\Big|^2\Big)^{1/2}$$
$$\le \int_{na}^{nb}v_n(s) \Big(\mathbb E \Big|\sum_{q=Q+1}^{2Q} \sum_{\ell=q-Q}^{Q} b_{q-\ell}^\eta a_{\ell} H_{k}(X_s)H_{\ell}(Y_s) \Big|^2\Big)^{1/2}ds$$
$$=\int_{na}^{nb}v_n(s)\Big( \sum_{q=Q+1}^{2Q} \sum_{\ell=q-Q}^{Q}\mathbb E \Big| b_{q-\ell}^\eta a_{\ell} H_{k}(X_s)H_{\ell}(Y_s) \Big|^2\Big)^{1/2}ds$$
$$\le \int_{na}^{nb}v_n(s) \Big(\mathbb E \Big|\xi^Q(s) \gamma^Q(s)-\xi^{[Q/3]}(s)\gamma^{[Q/3]}(s)\Big|^2\Big)^{1/2}ds.$$
Arguing as in \eqref{eqn:QLQ'L'}, the last display converges to $0$ as $Q\to \infty$. Henceforth, by \eqref{eqn:QLinfty} and by the triangle inequality, $\sum_{\ell+k \le Q} b_k^\eta a_{\ell}\int_{na}^{nb}H_{k}(X_s)H_{\ell}(Y_s)v_n(s)ds$ converges to $N^\eta([na,nb])$ in $L^2$ as $Q \to \infty$. Note that $a_k=0$ if $k$ is odd, thus this  completes the proof of (ii).
\endproof

It remains to justify the supporting lemmas.

\begin{proof}[Proof of Lemma \ref{l.dr}]
We follow the argument in   \cite{DHNgV}. Thanks to the lower bound for $V_n$ that follows from Lemma~\ref{lemma:3.3}, it suffices to show that the polynomial (of bounded degree) $\widetilde T_n$ has no double root in any given compact interval. By conditioning on $\xi_j,j\ge 1$, it is not hard to see that for any $t$ there is a constant $C$ that may depend on $t,n$ such that $\sup_{\gamma} P(|\widetilde T_n(t)-\gamma| <\epsilon)\le C\epsilon$ for any $\epsilon>0$.

 Now assume towards a contradiction that   there is some $t\in I:=[na,nb]$ such that $\widetilde T_n(t)=\frac{d}{dt}\widetilde T_n(t)=0$. We may divide $I$ into $O(1/\epsilon)$ subintervals of length at most $\epsilon$, and one such interval will contain $t$, and if $c$ is the center of this subinterval then  using the mean value theorem we can easily show that $|\widetilde T_n (c)|\le C \epsilon^2$ (for some constant $C$ that may depend on $n$ but independent of $\epsilon$).  However, using the union bound and the above small ball inequality, it follows that the probability that there is a subinterval (among the $O(1/\epsilon)$ intervals) whose center satisfy such estimate is $O(\epsilon^2)O(1/\epsilon)=O(\epsilon)$. Consequently the probability that there is a double root of $\widetilde T_n$ in $I$ is $O(\epsilon)$ for any $\epsilon>0$, which implies the desired claim.
\end{proof}

\begin{proof}[Proof of Lemma \ref{lemma:N_X}] The proof is similar to that of Lemma 4 in \cite{ADL}. For brevity, we write $X(t):=\overline{T}'_n(t)= \sum_{j=0}^n \xi_j q_j'(t)$. Our goal is to show that, for $N_X$ being the number of zeros of the process $X$ over $[na,nb]$, we have $\E N_X^2 <\infty$. For this, we need to consider the behavior of high derivatives of $X$. More specifically we use
\begin{align*}
\E (X^{(6)}(s) - X^{(6)}(t))^2 &= \E(\sum_{j=0}^n \xi_j (q_j^{(7)}(s)- q_j^{(7)}(t)))^2 =\sum_j (q_j^{(7)}(s)-q_j^{(7)}(t))^2\\
& = \sum_{j=0}^n (\int_{s}^t q_j^{(8)}(\eta)d\eta)^2 \le (\int_s^t (\sum_{j=0}^n q_j^{(8)}(\eta)^2)^{1/2}d\eta)^2 \\
&\le (s-t)^2 \max_{\eta \in (s,t)} \sum_{j=0}^n (q_j^{(8)}(\eta))^2 \le C(s-t)^2,
\end{align*}
where we used (v) of Lemma \ref{lemma:bounds} in the last estimate and Minkowski's inequality in the second to last estimate.
As a consequence, Dudley's theorem (see for instance \cite[Theorem 2.10]{AW}) applied to the Gaussian process $(X^{(6)}(t))_{t\in [na, nb]}$ (noting that as the metric is bounded by $C|s-t|$ by above, the covering number $N_\eps$ satisfies that $N_\eps \le \min\{C'n/\eps,1\}$, where $C'$ depends on $C,a,b$) implies that
$$\E \sup_{t\in [na ,nb]} X_t^{(6)} =O\left( \int_{0}^\infty \log (N_\eps)^{1/2}d\eps\right) =O(n).$$
We also note that the density of $X(t)$ is bounded since  the variance $\sum_{j=0}^n (q_j'(t))^2$ is bounded away from zero (by (i) of Lemma \ref{lemma:bounds}). We can then apply \cite[Theorem 3.6]{AW} to the process $X(t)$ with $m=2,p=5$, and deduce that $\E N_X^2 <\infty$.
\end{proof}

{\bf Acknowledgements.} We are grateful to F. Dalmao and D. Lubinsky for their very  helpful comments and suggestions. This work was initiated under the SQuaREs 2021 program of AIM,  we thank the Institute for the generous support. Y. Do is supported in part by NSF grant DMS-1800855.  H. H. Nguyen is supported by NSF CAREER grant DMS-1752345. O. Nguyen is supported by NSF grant DMS-2125031. I. E. Pritsker is supported in part by NSA grant H98230-21-1-0008, and by the Vaughn Foundation
endowed Professorship in Number Theory.

\appendix

\end{document}